%% file: eb-dft-v1.tex
\documentclass[11pt]{article}
\usepackage{fullpage,graphicx,amsmath,amsfonts,verbatim,subfig}
\usepackage{amssymb}
\usepackage{latexsym}
\usepackage{mathrsfs}
\usepackage{algorithm}
\usepackage{algorithmic}
\usepackage{multirow}

\input{defs.tex}

\title{A Unified Approach to Error Bounds for Structured Convex Optimization Problems}
\author{Zirui Zhou\thanks{Department of Systems Engineering and Engineering Management, The Chinese University of Hong Kong, Shatin, N.~T., Hong Kong.  E-mail: {\tt zrzhou@se.cuhk.edu.hk}} \and Anthony Man-Cho So\thanks{Department of Systems Engineering and Engineering Management, and, by courtesy, CUHK-BGI Innovation Institute of Trans-omics, The Chinese University of Hong Kong, Shatin, N.~T., Hong Kong.  E--mail: {\tt manchoso@se.cuhk.edu.hk}}}

\begin{document}
\maketitle

\begin{abstract}
Error bounds, which refer to inequalities that bound the distance of vectors in a test set to a given set by a residual function, have proven to be extremely useful in analyzing the convergence rates of a host of iterative methods for solving optimization problems.  In this paper, we present a new framework for establishing error bounds for a class of structured convex optimization problems, in which the objective function is the sum of a smooth convex function and a general closed proper convex function.  Such a class encapsulates not only fairly general constrained minimization problems but also various regularized loss minimization formulations in machine learning, signal processing, and statistics.  Using our framework, we show that a number of existing error bound results can be recovered in a unified and transparent manner. 
To further demonstrate the power of our framework, we apply it to a class of nuclear-norm regularized loss minimization problems and establish a new error bound for this class under a strict complementarity-type regularity condition. We then complement this result by constructing an example to show that the said error bound could fail to hold without the regularity condition.
Consequently, we obtain a rather complete answer to a question raised by Tseng.  We believe that our approach will find further applications in the study of error bounds for structured convex optimization problems.
%Furthermore, by assuming a strict complementarity-type regularity condition, we establish a new error bound for a class of nuclear norm-regularized loss minimization problems.  To complement this result, we construct an example to show that without the regularity condition, such an error bound could fail to hold. 
\end{abstract}

%\tableofcontents
%\newpage

\section{Introduction}
It has long been recognized that many convex optimization problems can be put into the form
\begin{equation}\label{eq:str-cvx-prob}
\min_{x\in\mathcal{E}} \left\{ F(x) := f(x) + P(x) \right\},
\end{equation}
where $\mathcal{E}$ is a finite-dimensional Euclidean space, $f:\mathcal{E}\limto(-\infty,+\infty]$ is a proper convex function that is continuously differentiable on $\mbox{int}(\mbox{dom}(f))$, and $P:\mathcal{E}\limto(-\infty,+\infty]$ is a closed proper convex function.  On one hand, the constrained minimization problem 
$$ \min_{x\in C} f(x), $$
where $C\subseteq\mathcal{E}$ is a closed convex set, is an instance of Problem~\eqref{eq:str-cvx-prob} with $P$ being the indicator function of $C$; \ie,
$$ P(x) = \left\{
\begin{array}{c@{\quad}l}
0 & \mbox{if } x \in C, \\
+\infty & \mbox{otherwise}.
\end{array}
\right.
$$
On the other hand, various data fitting problems in machine learning, signal processing, and statistics can be formulated as Problem~\eqref{eq:str-cvx-prob}, where $f$ is a loss function measuring the deviation of a solution from the observations and $P$ is a regularizer intended to induce certain structure in the solution.  With the advent of the big data era, instances of Problem~\eqref{eq:str-cvx-prob} that arise in contemporary applications often involve a large number of variables.  This has sparked a renewed interest in first-order methods for solving Problem~\eqref{eq:str-cvx-prob} in recent years; see, \eg,~\cite{tseng2010approximation,SNW12,W15} and the references therein. From a theoretical point of view, a fundamental issue concerning these methods is to determine their convergence rates. It is well known that various first-order methods for solving Problem~\eqref{eq:str-cvx-prob} will converge at the sublinear rate of $O(1/k^2)$, where $k\ge1$ is the number of iterations; see, \eg,~\cite{T08,beck2009fast,tseng2010approximation,N13}.  Moreover, the $O(1/k^2)$ convergence rate is optimal when the functions $f$ and $P$ are given by first-order oracles~\cite{nesterov2004introductory}.  However, in many applications, both $f$ and $P$ are given explicitly and have very specific structure.  It has been observed numerically that first-order methods for solving structured instances of Problem~\eqref{eq:str-cvx-prob} converge at a much faster rate than that suggested by the theory; see, \eg,~\cite{HYZ08,XZ13}. Thus, it is natural to ask whether the structure of the problem can be exploited in the convergence analysis to yield sharper convergence rate results.

As it turns out, a very powerful approach to addressing the above question is to study a so-called error bound property associated with Problem~\eqref{eq:str-cvx-prob}. Formally, let $\mathcal{X}\subseteq\mathcal{E}$ be the set of optimal solutions to Problem~\eqref{eq:str-cvx-prob}, assumed to be non-empty.  Furthermore, let $T\subseteq\mathcal{E}$ be a set satisfying $T\supseteq\mathcal{X}$ and $r:\mathcal{E}\limto\R_+$ be a function satisfying $r(x)=0$ if and only if $x\in\mathcal{X}$.  We say that Problem~\eqref{eq:str-cvx-prob} possesses a \emph{Lipschitzian error bound} (or simply \emph{error bound}) for $\mathcal{X}$ with \emph{test set} $T$ and \emph{residual function} $r$ if there exists a constant $\kappa>0$ such that 
\begin{equation} \label{eq:EB-def}
d(x,\mathcal{X})\le \kappa\cdot r(x) \quad\mbox{for all } x\in T,
\end{equation}
where $d(x,\mathcal{X}) := \inf_{z\in\mathcal{X}} \|z-x\|_2$ denotes the Euclidean distance from the vector $x\in\mathcal{E}$ to the set $\mathcal{X}\subseteq\mathcal{E}$; \cf~\cite{pang1997error}.  Conceptually, the error bound~\eqref{eq:EB-def} provides a handle on the structure of the objective function $F$ of Problem~\eqref{eq:str-cvx-prob} in the neighborhood $T$ of the optimal solution set $\mathcal{X}$ via the residual function $r$.  For the purpose of analyzing the convergence rates of first-order methods, one particularly useful choice of the residual function is $r_{\rm prox}(x):=\|R(x)\|_2$, where $R:\mathcal{E}\limto\mathcal{E}$ is the residual map defined by
\begin{equation}\label{eq:residual-map-def}
R(x):= \mbox{prox}_P(x-\nabla f(x)) - x
\end{equation}
and $\mbox{prox}_P:\mathcal{E}\limto\mathcal{E}$ is the proximal map associated with $P$; \ie,
\begin{equation}\label{eq:prox-oper}
 \mbox{prox}_P(x) := \argmin_{z\in\mathcal{E}} \left\{ \frac{1}{2}\| x - z \|_2^2 + P(z) \right\}. 
\end{equation}
Indeed, by comparing the optimality conditions of~\eqref{eq:str-cvx-prob} and~\eqref{eq:prox-oper}, it is immediate that $x\in\mathcal{X}$ if and only if $r_{\rm prox}(x)=0$.  Moreover, it is known that many first-order methods for solving Problem~\eqref{eq:str-cvx-prob} have update rules that aim at reducing the value of the residual function; see, \eg,~\cite{LT93,CW05,tseng2009coordinate}.  This leads to the following instantiation of~\eqref{eq:EB-def}:

%With the residual function $r_{\rm prox}$, we can now consider the following instantiation of~\eqref{eq:EB-def}:

\medskip
\noindent
{\bf Error Bound with Proximal Map-Based Residual Function.} For any $\zeta\ge v^*:=\min_{x\in\mathcal{E}} F(x)$, there exist constants $\kappa>0$ and $\epsilon>0$ such that
\begin{equation} \tag{EBP} \label{eq:err-bd}
d(x,\mathcal{X}) \leq \kappa \|R(x)\|_2 \quad\mbox{for all } x\in\mathcal{E} \mbox{ with } F(x)\leq \zeta, \, \|R(x)\|_2 \le \epsilon.
\end{equation}

\noindent The usefulness of the error bound~\eqref{eq:err-bd} comes from the fact that whenever it holds, a host of first-order methods for solving Problem~\eqref{eq:str-cvx-prob}, such as the proximal gradient method, the extragradient method, and the coordinate (gradient) descent method, can be shown to converge linearly; see~\cite{LT93,tseng2010approximation} and the references therein.  Thus, an important research issue is to identify conditions on the functions $f$ and $P$ under which the error bound~\eqref{eq:err-bd} holds.  Nevertheless, despite the efforts of many researchers over a long period of time, the repertoire of instances of Problem~\eqref{eq:str-cvx-prob} that are known to possess the error bound~\eqref{eq:err-bd} is still rather limited.  Below are some representative scenarios in which~\eqref{eq:err-bd} has been shown to hold:

%Such a property is closely tied to the structure of Problem~\eqref{eq:str-cvx-prob} and proves key to establishing much improved convergence rate results for a host of first-order methods.  
%
%The error bound defined above has rich and diverse applications in mathematical programming. When the residual function is easily computable, error bounds can provide stopping rules in the practical implementation of iterative methods. Furthermore, it has been widely recognized that if the residual function is related to the algorithmic mapping of the iterative methods for solving optimization problems, such error bound is the key to analyzing the convergence rates of these numerical methods, especially in treating problems with degenerate solutions~\cite{luo1993convergence, pang1997error}. 
%
\begin{itemize}
\item[(S1).] (\cite[Theorem 3.1]{pang1987posteriori}) $\mathcal{E}=\R^n$, $f$ is strongly convex, $\nabla f$ is Lipschitz continuous, and $P$ is arbitrary (but closed, proper, and convex).

\item[(S2).] (\cite[Theorem 2.1]{luo1992linear}) $\mathcal{E}=\mathbb{R}^n$, $f$ takes the form $f(x) = h(Ax) + \langle c,x\rangle$, where $A\in\mathbb{R}^{m\times n}$, $c\in\mathbb{R}^n$ are given and $h:\R^m\limto(-\infty,+\infty]$ is proper and convex with the properties that (i) $h$ is continuously differentiable on $\mbox{dom}(h)$, assumed to be non-empty and open, and (ii) $h$ is strongly convex and $\nabla h$ is Lipschitz continuous on any compact subset of $\mbox{dom}(h)$, and $P$ has a polyhedral epigraph.

%\item[(S3).] (\cite[Theorem 4.1]{luo1993convergence}) $\mathcal{E}=\R^n$, $f$ takes the form $f(x)=\max_{y\in Y} \left\{ \langle y,Ax \rangle - h(y) \right\} + \langle c,x \rangle$, where $A\in\R^{m\times n}$, $c\in\R^n$, and $h:\R^m\limto(-\infty,\infty]$ are as in scenario (S2), $Y\subseteq\R^m$ is polyhedral, and $P$ has a polyhedral epigraph.

\item[(S3).] (\cite[Theorem 2]{tseng2010approximation}; \cf~\cite[Theorem 1]{ZJL13}) $\mathcal{E}=\mathbb{R}^n$, $f$ takes the form $f(x)=h(Ax)$, where $A\in\R^{m\times n}$ and $h:\R^m\limto(-\infty,+\infty]$ are as in scenario (S2), and $P$ is the grouped LASSO regularizer; \ie, $P(x) = \sum_{J\in\mathcal{J}}\omega_J\|x_J\|_2$, where $\mathcal{J}$ is a partition of the index set $\{1,\ldots,n\}$, $x_J\in\R^{|J|}$ is the vector obtained by restricting $x\in\R^n$ to the entries in $J\in\mathcal{J}$, and $\omega_J\ge0$ is a given parameter.
\end{itemize}
In many applications, such as regression problems, the function $f$ of interest is not strongly convex but has the structure described in scenarios (S2) and (S3).  However, a number of widely used structure-inducing regularizers $P$---most notably the nuclear norm regularizer---are not covered by these scenarios. One of the major difficulties in establishing the error bound~\eqref{eq:err-bd} for regularizers other than those described in scenarios (S2) and (S3) is that they typically have non-polyhedral epigraphs.  Moreover, existing approaches to establishing the error bound~\eqref{eq:err-bd} are quite ad hoc in nature and cannot be easily generalized.  Thus, in order to identify more scenarios in which the error bound~\eqref{eq:err-bd} holds, some new ideas would seem to be necessary.

% In particular, they do not offer much insight into the conditions on Problem~\eqref{eq:str-cvx-prob} that are necessary and/or sufficient for the error bound to hold, thereby limiting their generalizability.

In this paper, we present a new analysis framework for studying the error bound property~\eqref{eq:err-bd} associated with Problem~\eqref{eq:str-cvx-prob}.  The framework applies to the setting where $f$ has the form described in scenario (S2) and $P$ is \emph{any} closed proper convex function.  In particular, it applies to all the scenarios (S1)--(S3).  Our first contribution is to elucidate the relationship between the error bound property~\eqref{eq:err-bd} and various notions in set-valued analysis.  This allows us to utilize powerful tools from set-valued analysis to elicit the key properties of Problem~\eqref{eq:str-cvx-prob} that can guarantee the validity of~\eqref{eq:err-bd}.  Specifically, we show that the problem of establishing the error bound~\eqref{eq:err-bd} can be reduced to that of checking the calmness of a certain set-valued mapping $\Gamma$ induced by the optimal solution set $\mathcal{X}$ of Problem~\eqref{eq:str-cvx-prob}; see Corollary~\ref{cor:orig-eb}.  Furthermore, using the fact that $\mathcal{X}$ can be expressed as the intersection of a polyhedron and the inverse of the subdifferential of $P$ at a certain point $-\bar{g}\in\mathcal{E}$ (see Proposition~\ref{prop:opt-invariant}), we show that the calmness of $\Gamma$ is in turn implied by (i) the bounded linear regularity of the two intersecting sets and (ii) the calmness of $(\del P)^{-1}$ at $-\bar{g}$; see Theorem~\ref{thm:two-conditions}.  These results provide a concrete starting point for verifying the error bound property~\eqref{eq:err-bd} and make it possible to simplify the analysis substantially.  We remark that when $P$ has a polyhedral epigraph, the early works~\cite{LT93,luo1993convergence} of Luo and Tseng have already pointed out a connection between~\eqref{eq:err-bd} and the calmness of certain polyhedral multi-function.  However, such an idea has not been further explored in the literature to tackle more general forms of $P$.

To demonstrate the power of our proposed framework, we apply it to scenarios (S1)--(S3) and show that the error bound results in~\cite{pang1987posteriori,luo1992linear,tseng2010approximation} can be recovered in a unified manner; see Sections~\ref{subsec:str-cvx}--\ref{subsec:groupLASSO}. It is worth noting that scenario (S3) involves the non-polyhedral grouped LASSO regularizer, and the existing proof of the validity of the error bound~\eqref{eq:err-bd} in this scenario employs a highly intricate argument~\cite{tseng2010approximation}.  By contrast, our approach leads to a much simpler and more transparent proof.  % of the result in~\cite{tseng2010approximation} concerning the validity of the error bound~\eqref{eq:err-bd} in this scenario.
Motivated by the above success, we proceed to apply our framework to the following scenario, which again involves a non-polyhedral regularizer and arises in the context of low-rank matrix optimization:
\begin{itemize}
\item[(S4).] $\mathcal{E}=\mathbb{R}^{m\times n}$, $f$ takes the form $f(X) = h(\mathcal{A}(X)) + \langle C,X \rangle$, where $\mathcal{A}:\mathbb{R}^{m\times n}\rightarrow\mathbb{R}^\ell$ is a given linear operator, $C\in\R^{m\times n}$ is a given matrix, $h:\mathbb{R}^\ell\limto(-\infty,+\infty]$ is as in scenario (S2), and $P$ is the nuclear norm regularizer; \ie, $P(X) = \|X\|_{*}$.  
\end{itemize}
The validity of the error bound~\eqref{eq:err-bd} in this scenario was left as an open question in~\cite{tseng2010approximation} and to date is still unresolved.\footnote{It was claimed in~\cite{HZSL13} that the error bound~\eqref{eq:err-bd} holds in scenario (S4).  However, there is a critical flaw in the proof. Specifically, contrary to what was claimed in~\cite[Supplementary Material, Section C]{HZSL13}, the matrices $U^k$ and $V^k$ that satisfy displayed equations~(37) and~(38) need not satisfy displayed equation~(35). The erroneous claim was due to an incorrect application of~\cite[Lemma 4.3]{sun2002strong}. We thank Professor Defeng Sun and Ms.~Ying Cui for bringing this issue to our attention.}  As our second contribution in this work, we show that under a strict complementarity-type regularity condition on the optimal solution set $\mathcal{X}$ of Problem~\eqref{eq:str-cvx-prob}, the error bound~\eqref{eq:err-bd} holds in scenario (S4); see Proposition~\ref{prop:nucnorm-eb}.  This is achieved by verifying conditions (i) and (ii) mentioned in the preceding paragraph.  Specifically, we first show that condition (i) is satisfied under the said regularity condition.  Then, we prove that $(\del\|\cdot\|_*)^{-1}$ is calm everywhere, which implies that condition (ii) is always satisfied; see Proposition~\ref{prop:nucnorm-metric-subreg}.  We note that to the best of our knowledge, this last result is new and could be of independent interest.  To further understand the role of the regularity condition, we demonstrate via a concrete example that without such condition, the error bound~\eqref{eq:err-bd} could fail to hold; see Section~\ref{subsubsec:nucnorm-eb}. Consequently, we obtain a rather complete answer to the question raised by Tseng~\cite{tseng2010approximation}.

%We show that in scenarios (S1)--(S3), both conditions (i) and (ii) mentioned in the preceding paragraph are satisfied; see Sections~\ref{subsec:str-cvx}--\ref{subsec:groupLASSO}.  Consequently, we recover the error bound results in~\cite{pang1987posteriori,luo1992linear,tseng2010approximation} in a unified manner.  It is also worth noting that our approach gives a much simpler and more intuitive proof of the result in~\cite{tseng2010approximation} concerning the validity of the error bound~\eqref{eq:err-bd} in scenario (S3).  

%From the above discussion, we see that our proposed framework not only provides a unified treatment of a number of existing error bound results for Problem~\eqref{eq:str-cvx-prob} but also has the potential to .  We expect that our framework will find further applications in the study of the error bound property~\eqref{eq:err-bd} associated with the structured convex optimization problem~\eqref{eq:str-cvx-prob}.
%unify existing and can potentially be generalized to other non-polyhedral

The following notations will be used throughout the paper.  Let $\mathcal{E},\mathcal{T}$ denote finite-dimensional Euclidean spaces.  The closed ball around $\bar{x}\in\mathcal{E}$ with radius $r>0$ in $\mathcal{E}$ is given by $\mathbb{B}_{\mathcal{E}}(\bar{x},r):=\left\{ x\in\mathcal{E} \mid \|x-\bar{x}\|_2 \le r\right\}$. For simplicity, we denote the closed unit ball $\mathbb{B}_{\mathcal{E}}(\bz,1)$ in $\mathcal{E}$ by $\mathbb{B}_{\mathcal{E}}$.  We use $\mathbb{S}^n$ and $\mathbb{O}^n$ to denote the sets of $n\times n$ real symmetric matrices and $n\times n$ orthogonal matrices, respectively.  Given a matrix $X\in\mathbb{S}^n$, we use $X\in\mathbb{S}_+^n$ or $X\succeq\bz$ (resp.~$X\in\Sy_{++}^n$ or $X\succ\bz$) to indicate that $X$ is positive semidefinite (resp.~positive definite).  Also, we use $\|X\|_F$ and $\|X\|$ to denote the Frobenius norm and spectral norm of the matrix $X\in\R^{m\times n}$, respectively.

\section{Preliminaries}
\subsection{Basic Setup} \label{subsec:setup}
Consider the optimization problem~\eqref{eq:str-cvx-prob}.  Recall that its optimal value and optimal solution set are denoted by $v^{*}$ and $\mathcal{X}$, respectively.  We shall make the following assumptions in our study:
\begin{ass}\label{ass:smooth-func} (Structural Properties of the Objective Function)
\begin{enumerate}
\item[\subpb] The function $f:\mathcal{E}\limto(-\infty,+\infty]$ takes the form
\begin{equation}\label{eq:form-of-f}
f(x) = h(\mathcal{A}(x)) + \langle c,x\rangle,
\end{equation}
where $\mathcal{A}:\mathcal{E}\rightarrow\mathcal{T}$ is a linear operator, $c \in \mathcal{E}$ is a given vector, and $h:\mathcal{T}\limto(-\infty,+\infty]$ is a convex function with the following properties:
\begin{enumerate}
\item[(i)] The effective domain ${\rm dom}(h)$ of $h$ is non-empty and open, and $h$ is continuously differentiable on ${\rm dom}(h)$.

\item[(ii)] For any compact convex set $V\subseteq{\rm dom}(h)$, the function $h$ is strongly convex and its gradient $\nabla h$ is Lipschitz continuous on $V$.
\end{enumerate}

\item[\subpb] The function $P:\mathcal{E}\limto(-\infty,+\infty]$ is convex, closed, and proper.
\end{enumerate}
\resetspb
\end{ass}
\begin{ass}\label{ass:level-bd} (Properties of the Optimal Solution Set) The optimal solution set $\mathcal{X}$ is non-empty and compact.  In particular, $v^*>-\infty$.
% For any $\zeta\geq v^{*}$, the level set $L(\zeta):=\{x\in\mathcal{E}\mid F(x)\leq \zeta\}$ is a compact subset of~$\mathcal{E}$.
\end{ass}

The above assumptions yield several useful consequences.  First, Assumption~\ref{ass:smooth-func}(a-i) implies that $\mbox{dom}(f)=\mathcal{A}^{-1}(\mbox{dom}(h)) = \left\{ x\in\mathcal{E} \mid \mathcal{A}(x) \in \mbox{dom}(h) \right\}$ is also non-empty and open, and $f$ is continuously differentiable on $\mbox{dom}(f)$.  Second, under Assumption~\ref{ass:smooth-func}(a-ii), if the Lipschitz constant of $\nabla h$ on the compact convex set $V \subseteq \mbox{dom}(h)$ is $L_h(V)$, then the Lipschitz constant of $\nabla f$ on $\mathcal{A}^{-1}(V)$ is at most $L_h(V)\|\mathcal{A}\|^2$, where $\|\mathcal{A}\|$ is the spectral norm of $\mathcal{A}$. Third, Assumption~\ref{ass:smooth-func} implies that $F$ is a closed proper convex function.  Together with Assumption~\ref{ass:level-bd} and~\cite[Corollary 8.7.1]{rockafellar1970convex}, we conclude that for any $\zeta \ge v^*$, the level set $L(\zeta):=\left\{ x\in\mathcal{E} \mid F(x) \le \zeta \right\}$ is a compact subset of $\mathcal{E}$.

Assumptions~\ref{ass:smooth-func} and~\ref{ass:level-bd} are automatically satisfied by a number of applications. As an illustration, consider the problem of regularized empirical risk minimization of linear predictors, which underlies much of the development in machine learning.  With $N$ being the number of data points and $\mathcal{T}=\R^N$, the problem takes the form
\begin{equation} \label{eq:reg-ERM}
\min_{x\in\mathcal{E}} \left\{ \frac{1}{N}\sum_{i=1}^N \ell \left( [\mathcal{A}(x)]_i,b_i \right) + P(x) \right\},
\end{equation}
where $[\mathcal{A}(x)]_i$ is the $i$-th component of the vector $\mathcal{A}(x)$ and represents the $i$-th linear prediction, $b_i$ is the $i$-th response, $\ell:\R\limto\R$ is a smooth convex loss function, and $P:\mathcal{E}\limto(-\infty,+\infty]$ is a regularizer used to induce certain structure in the solution.  It is clear that Problem~\eqref{eq:reg-ERM} is an instance of Problem~\eqref{eq:str-cvx-prob}.  Moreover, one can easily verify that when instantiated with the loss functions and regularizers in Table~\ref{tab:loss-reg}---which have been widely used in the machine learning literature---Problem~\eqref{eq:reg-ERM} satisfies both Assumptions~\ref{ass:smooth-func} and~\ref{ass:level-bd}.

\begin{table}[htb]
\centering
\subfloat[Loss Functions\label{subtab:loss}]{
\centering
\begin{tabular}{c|c|c}
  %\hline
  & $\ell(y,b)$ & Domain of $b$ \\ \hline\hline
  Linear Regression & $\begin{array}{c} \vspace{-10pt} \\ \displaystyle{ \frac{1}{2}(y-b)^2 } \\ \vspace{-10pt} \end{array}$ & $\R$ \\\hline
  Logistic Regression & $\log(1+\exp(-yb))$ & $\{-1,1\}$ \\ \hline
%  loss associated & \multirow{3}{*}{$yb-\log y$} & \multirow{3}{*}{$[0,\infty)$} \\ 
%  with exponential & & \\
%  noise~\cite{sardy2004automatic} & & \\\hline
  Poisson Regression & $-yb+\exp(y)$ & $\{0,1,\ldots\}$
\end{tabular}
}
\\
\subfloat[Regularizers\label{subtab:reg}]{
\centering
\begin{tabular}{c|c|c}
& $\mathcal{E}$ & $P(x)$ \\ \hline\hline
LASSO & $\R^n$ & $\|x\|_1$ \\ \hline
Ridge & $\R^n$ & $\begin{array}{c} \vspace{-12pt} \\ \|x\|_2^2 \\ \vspace{-12pt}\end{array}$ \\ \hline
Grouped LASSO & $\R^n$ & $\begin{array}{c} \vspace{-10pt} \\ \displaystyle{\sum_{J\in\mathcal{J}} \omega_J\|x_J\|_2}, \, \omega_J \ge 0, \\ \mathcal{J} \mbox{ a partition of }\{1,\ldots,n\} \\ \vspace{-12pt} \end{array}$ \\ \hline
Nuclear Norm & $\R^{m\times n}$ & $\|x\|_*$
\end{tabular}
}
\caption{Some commonly used loss functions and regularizers.}\label{tab:loss-reg}
\end{table}

\subsection{A Characterization of the Optimal Solution Set $\mathcal{X}$}
Since Problem~\eqref{eq:str-cvx-prob} is an unconstrained convex optimization problem, its first-order optimality condition is both necessary and sufficient for optimality.  Hence, we have 
\begin{equation} \label{eq:first-order-opt}
\mathcal{X} = \left\{ x\in\mathcal{E} \mid \bz \in \nabla f(x)+\del P(x) \right\}.
\end{equation}
The following proposition shows that under Assumptions~\ref{ass:smooth-func} and~\ref{ass:level-bd}, the optimal solution set $\mathcal{X}$ admits an alternative, more explicit characterization.  Such a characterization will be central to our analysis of the error bound property associated with Problem~\eqref{eq:str-cvx-prob}. 
%% --old--- As the smooth function $f$ of \eqref{eq:str-cvx-prob} is assumed to be of form \eqref{eq:form-of-f} and satisfy Assumption \ref{ass:smooth-func}, the optimal solution set $\mathcal{X}$ has the following invariant property.
\begin{prop}\label{prop:opt-invariant}
Consider the optimization problem~\eqref{eq:str-cvx-prob}.  Under Assumptions \ref{ass:smooth-func} and \ref{ass:level-bd}, there exists a $\bar{y}\in\mathcal{T}$ such that
\begin{equation} \label{eq:invar}
\mathcal{A}(x) = \bar{y}, \,\,\, \nabla f(x) = \bar{g} \quad\mbox{for all } x\in\mathcal{X},
\end{equation}
where $\bar{g} = \mathcal{A}^{*}\nabla h(\bar{y}) + c \in \mathcal{E}$.  In particular, we have
\begin{equation}\label{eq:optimal-set}
\mathcal{X} = \left\{ x\in\mathcal{E} \mid \mathcal{A}(x) = \bar{y},\, -\bar{g} \in \partial P(x) \right\}.
\end{equation}
\end{prop}
\begin{proof}
The proof of~\eqref{eq:invar} is rather standard; \cf~\cite{T91,luo1992linear}.  For completeness' sake, we include the proof here.  For arbitrary $x_1,x_2\in\mathcal{X}$, let $y_1 = \mathcal{A}(x_1)$ and $y_2 = \mathcal{A}(x_2)$.  Note that the line segment between $y_1$ and $y_2$ is a compact convex subset of $\mbox{dom}(h)$.  By Assumption \ref{ass:smooth-func}(a-ii), the function $h$ is strongly convex on this set.  Thus, there exists a $\sigma>0$ such that 
$$ h\left(\frac{y_1+y_2}{2}\right) \leq \frac{1}{2} h(y_1) + \frac{1}{2} h(y_2) - \frac{\sigma}{2}\|y_1 - y_2\|_2^2. $$
Due to \eqref{eq:form-of-f}, the above is equivalent to
$$ f\left(\frac{x_1+x_2}{2}\right) \leq \frac{1}{2} f(x_1) + \frac{1}{2} f(x_2) - \frac{\sigma}{2}\|y_1 - y_2\|_2^2. $$
Moreover, the convexity of $P$ gives
$$ P\left(\frac{x_1+x_2}{2}\right) \leq \frac{1}{2} P(x_1) + \frac{1}{2} P(x_2). $$
Upon adding the above two inequalities and using $x_1,x_2\in\mathcal{X}$, we have
$$ F\left(\frac{x_1+x_2}{2}\right) \leq v^{*} - \frac{\sigma}{2}\|y_1 - y_2\|_2^2. $$
This implies that $y_1=y_2$, for otherwise the above inequality contradicts the fact that $v^*$ is the optimal value of Problem~\eqref{eq:str-cvx-prob}.  Consequently, the map $x\mapsto\mathcal{A}(x)$ is invariant over $\mathcal{X}$; \ie, there exists a $\bar{y}\in\mathcal{T}$ such that $\mathcal{A}(x)=\bar{y}$ for all $x\in\mathcal{X}$.  Now, using~\eqref{eq:form-of-f} and Assumption~\ref{ass:smooth-func}(a-i), we compute $\nabla f(x) = \mathcal{A}^{*}\nabla h(\mathcal{A}(x))+c$.  Since $\mathcal{A}(x)=\bar{y}$ for all $x\in\mathcal{X}$, we have $\nabla f(x) = \mathcal{A}^{*}\nabla h(\bar{y})+c=\bar{g}$ for all $x\in\mathcal{X}$.  This completes the proof of~\eqref{eq:invar}.

To establish~\eqref{eq:optimal-set}, we first observe that by~\eqref{eq:first-order-opt} and~\eqref{eq:invar}, every $x\in\mathcal{X}$ belongs to the set on the right-hand side of~\eqref{eq:optimal-set}.  Now, for any $x\in\mathcal{E}$ satisfying $\mathcal{A}(x) = \bar{y}$ and $-\bar{g}\in\partial P(x)$, we can use the relationships $\bar{g} = \mathcal{A}^{*}\nabla h(\bar{y})+c$ and $\nabla f(x) = \mathcal{A}^{*}\nabla h(\mathcal{A}(x))+c$ to get $\mathbf{0}\in\nabla f(x) + \partial P(x)$.  This, together with~\eqref{eq:first-order-opt}, implies that $x\in\mathcal{X}$, as desired.
\end{proof}

\subsection{Tools from Set-Valued Analysis} \label{subsec:sva}
Proposition~\ref{prop:opt-invariant} reveals that the optimal solution set $\mathcal{X}$ of Problem~\eqref{eq:str-cvx-prob} is completely characterized by the vectors $\bar{y}\in\mathcal{T}$ and $\bar{g}\in\mathcal{E}$.  Thus, in order to estimate $d(x,\mathcal{X})$ for some $x\in\mathcal{E} \setminus \mathcal{X}$, a natural idea is to take $y=\mathcal{A}(x)$ and an arbitrary $-g\in\del P(x)$ and establish a relationship between $d(x,\mathcal{X})$ and $\|(y,g)-(\bar{y},\bar{g})\|_2$.  Intuitively, if $\mathcal{X}$ is ``nice'' (\eg, satisfies certain regularity condition), then one should be able to control the (local) growth of $d(x,\mathcal{X})$ by that of a ``nice'' function of $\|(y,g)-(\bar{y},\bar{g})\|_2$.  Such an idea can be formalized using tools from set-valued analysis, which we now introduce.

Let $\mathcal{E}_1$ and $\mathcal{E}_2$ be finite-dimensional Euclidean spaces.  We say that a mapping $\Gamma$ is a \emph{multi-function} (or \emph{set-valued mapping}) from $\mathcal{E}_1$ to $\mathcal{E}_2$ (denoted by $\Gamma: \mathcal{E}_1\rightrightarrows\mathcal{E}_2$) if it assigns a subset $\Gamma(u)$ of $\mathcal{E}_2$ to each vector $u\in\mathcal{E}_1$.  The \emph{graph} and \emph{domain} of $\Gamma$ are defined by
\begin{eqnarray*}
\mbox{gph}(\Gamma) &:=&\left\{ (u,v) \in \mathcal{E}_1\times\mathcal{E}_2 \mid v \in \Gamma(u) \right\}, \\
\mbox{dom}(\Gamma)&:=& \left\{ u\in\mathcal{E}_1 \mid \Gamma(u) \neq \emptyset \right\},
\end{eqnarray*}
respectively.  The inverse mapping of $\Gamma$, denoted by $\Gamma^{-1}$, is the multi-function from $\mathcal{E}_2$ to $\mathcal{E}_1$ defined by
$$%\begin{equation}\label{eq:def-inverse}
\Gamma^{-1}(v) := \left\{ u\in\mathcal{E}_1 \mid v \in \Gamma(u)\right\}. 
$$%\end{equation}
Before we proceed further, let us briefly illustrate some of the concepts above.

\medskip
\noindent{\bf Example}
\begin{enumerate} 
\item[\subpb] Let $A\in\R^{m\times n}$ be a given $m\times n$ matrix.  The mapping $\Gamma$ defined by $\Gamma(b)=\left\{ x\in\mathbb{R}^n \mid Ax = b\right\}$ is a multi-function from $\R^m$ to $\R^n$.  Here, $\Gamma(b)$ is simply the solution set of the linear system $Ax=b$. 

\item[\subpb] Let $P:\mathcal{E}\limto(-\infty,+\infty]$ be a closed proper convex function. Its subdifferential $\del P$ is a multi-function from $\mathcal{E}$ to $\mathcal{E}$.  Moreover, by~\cite[Corollary 23.5.1]{rockafellar1970convex}, we have $(\partial P)^{-1} = \partial P^{*}$, where $P^{*}$ is the conjugate of $P$. \\ \endproof
\end{enumerate}
\resetspb

%\medskip
Next, we introduce two regularity notions regarding set-valued mappings.
% that will be useful in our investigation.
\begin{defi}\label{defi:calmness-metric-sub}
(see, \eg,~\cite[Chapter 3H]{dontchev2009implicit}) 
\begin{enumerate}
\item[\subpb] A multi-function $\Gamma:\mathcal{E}_1\rightrightarrows\mathcal{E}_2$ is said to be \emph{calm} at $\bar{u}\in\mathcal{E}_1$ for $\bar{v}\in\mathcal{E}_2$ if $(\bar{u},\bar{v})\in{\rm gph}(\Gamma)$ and there exist constants $\kappa, \epsilon>0$ such that 
\begin{equation}\label{eq:calmness}
\Gamma(u) \cap \mathbb{B}_{\mathcal{E}_2}(\bar{v},\epsilon) \subseteq \Gamma(\bar{u}) + \kappa \|u - \bar{u}\|_2 \mathbb{B}_{\mathcal{E}_2} \quad \mbox{for all } u\in\mathcal{E}_1. % \mathbb{B}_{\mathcal{E}_1}(\bar{u},\delta).
\end{equation}

\item[\subpb] A multi-function $\Lambda:\mathcal{E}_1\rightrightarrows\mathcal{E}_2$ is said to be \emph{metrically sub-regular} at $\bar{u}\in\mathcal{E}_1$ for $\bar{v}\in\mathcal{E}_2$ if $(\bar{u},\bar{v})\in{\rm gph}(\Lambda)$ and there exist constants $\kappa, \epsilon>0$ such that
\begin{equation}\label{eq:metric-sub}
d\left( u, \Lambda^{-1}(\bar{v}) \right) \leq \kappa \cdot d\left( \bar{v}, \Lambda(u) \right) \quad \mbox{for all } u\in\mathbb{B}_{\mathcal{E}_1}(\bar{u},\epsilon). 
\end{equation}
\end{enumerate}
\resetspb
\end{defi}
The notions of calmness and metric sub-regularity have played a central role in the study of error bounds; see, \eg,~\cite{pang1997error,F02,S06,P10,K15} and the references therein.  To see what these notions would yield in the context of Problem~\eqref{eq:str-cvx-prob}, consider the multi-function $\Gamma:\mathcal{T}\times\mathcal{E}\rightrightarrows\mathcal{E}$ given by
\begin{equation} \label{eq:def-sol-map}
\Gamma(y,g) = \left\{ x\in\mathcal{E} \mid \mathcal{A}(x)=y, \, -g\in\del P(x) \right\}.
\end{equation}
Suppose that $\Gamma$ is calm at $(\bar{y},\bar{g})\in\mathcal{T}\times\mathcal{E}$ for $\bar{x}\in\mathcal{E}$, where $\bar{y}$ and $\bar{g}$ are given in Proposition~\ref{prop:opt-invariant}.  Note that $\mathcal{X}=\Gamma(\bar{y},\bar{g})$ and $\bar{x}\in\mathcal{X}$.  Hence, by~\eqref{eq:calmness}, there exist constants $\kappa,\epsilon>0$ such that
\begin{equation} \label{eq:calm-eb}
d(x,\mathcal{X}) \le \kappa\|(y,g)-(\bar{y},\bar{g})\|_2 \quad\mbox{for all } x\in\Gamma(y,g) \cap \mathbb{B}_{\mathcal{E}}(\bar{x},\epsilon) \mbox{ and } (y,g) \in \mathcal{E}_1. % (y,g) \in \mathbb{B}_{\mathcal{T}\times\mathcal{E}}((\bar{y},\bar{g}),\delta). 
\end{equation}
Since $x\in\Gamma(y,g) \cap \mathbb{B}_{\mathcal{E}}(\bar{x},\epsilon)$ is equivalent to $(y,g)\in\Gamma^{-1}(x)$ and $x\in\mathbb{B}_{\mathcal{E}}(\bar{x},\epsilon)$, it follows from~\eqref{eq:calm-eb} that
%Now, set $\eta=\min\{\epsilon,\kappa\delta\}>0$ and let $x\in\mathbb{B}_{\mathcal{E}}(\bar{x},\eta)$.  Suppose that $(y,g)\in\Gamma^{-1}(x)$.  If $(y,g)\in\mathbb{B}_{\mathcal{T}\times\mathcal{E}}((\bar{y},\bar{g}),\delta)$, then $d(x,\mathcal{X}) \le \kappa\|(y,g)-(\bar{y},\bar{g})\|_2$ by~\eqref{eq:calm-eb}.  Otherwise, we have $d(x,\mathcal{X}) \le \|x-\bar{x}\|_2 \le \eta \le \kappa\delta \le \kappa\|(y,g)-(\bar{y},\bar{g})\|_2$.  It follows that
\begin{equation} \label{eq:subreg-eb}
d(x,\mathcal{X}) \le \kappa \cdot d\left( (\bar{y},\bar{g}), \Gamma^{-1}(x) \right) \quad\mbox{for all } x\in\mathbb{B}_{\mathcal{E}}(\bar{x},\epsilon),
\end{equation}
which is an error bound for $\mathcal{X}$ with test set $\mathbb{B}_{\mathcal{E}}(\bar{x},\epsilon)$ and residual function $x\mapsto d\left( (\bar{y},\bar{g}), \Gamma^{-1}(x) \right)$.  Incidentally, the inequality~\eqref{eq:subreg-eb} also shows that the multi-function $\Gamma^{-1}:\mathcal{E}\rightrightarrows\mathcal{T}\times\mathcal{E}$ is metrically sub-regular at $\bar{x}\in\mathcal{E}$ for $(\bar{y},\bar{g})\in\mathcal{T}\times\mathcal{E}$.

The error bound~\eqref{eq:subreg-eb} shows that under a calmness assumption on the multi-function $\Gamma$ given in~\eqref{eq:def-sol-map}, the local growth of $d(x,\mathcal{X})$ is on the order of $\|(y,g)-(\bar{y},\bar{g})\|_2$, where $y=\mathcal{A}(x)$ and $-g\in\del P(x)$ is arbitrary.  This realizes the idea mentioned at the beginning of this sub-section.  However, we are ultimately interested in establishing the error bound~\eqref{eq:err-bd}, which is concerned with the test set $\left\{ x\in\mathcal{E} \mid F(x) \le \zeta, \, \|R(x)\|_2 \le \epsilon \right\}$ (where $\zeta\ge v^*$ is arbitrary and $\epsilon=\epsilon(\zeta)$ depends on $\zeta$) and residual function $x\mapsto\|R(x)\|_2$.  At first sight, it is not clear whether the error bounds~\eqref{eq:err-bd} and~\eqref{eq:subreg-eb} are compatible.  Indeed, the former involves only easily computable quantities (\ie, $F(x)$ and $\|R(x)\|_2$), while the latter involves quantities that are generally not known \emph{a priori} (\ie, $\bar{x}\in\mathcal{E}$, $\bar{y}\in\mathcal{T}$, and $\bar{g}\in\mathcal{E}$).  Nevertheless, as we shall demonstrate in Section~\ref{sec:charac-eb}, the latter can be used to establish the former under some mild conditions.

Before we leave this section, let us record two useful results regarding the notions of calmness and metric sub-regularity.  The first is a well-known equivalence between the calmness of a multi-function and the metric sub-regularity of its inverse.  One direction of the equivalence has already manifested in our discussion above.
\begin{fact}\label{fact:calm-metricsub} (see, \eg,~\cite[Theorem 3H.3]{dontchev2009implicit}) For a multi-function $\Gamma:\mathcal{E}_1\rightrightarrows\mathcal{E}_2$, let $(\bar{u},\bar{v})\in{\rm gph}(\Gamma)$. Then, $\Gamma$ is calm at $\bar{u}$ for $\bar{v}$ if and only if its inverse $\Gamma^{-1}$ is metrically sub-regular at $\bar{v}$ for $\bar{u}$.
\end{fact}

The second result concerns a multi-function $\Gamma:\mathcal{E}_1\rightrightarrows\mathcal{E}_2$ that is calm at $\bar{u}\in\mathcal{E}_1$ for a set of points $\bar{V} \subseteq \Gamma(\bar{u})$.  It shows that if $\bar{V}$ is compact, then the neighborhoods around each $\bar{v}\in\bar{V}$ in the definition of calmness can be made uniform.
\begin{prop}\label{prop:ulc-calm}
For a multi-function $\Gamma:\mathcal{E}_1\rightrightarrows\mathcal{E}_2$, let $\bar{u}\in{\rm dom}(\Gamma)$ and suppose that $\bar{V}\subseteq\Gamma(\bar{u})$ is compact. Then, the following statements are equivalent:
\begin{enumerate}
\item[\subpb] $\Gamma$ is calm at $\bar{u}$ for any $\bar{v}\in\bar{V}$.

\item[\subpb] There exist constants $\kappa,\epsilon>0$ such that
$$
\Gamma(u) \cap \left( \bar{V}+\epsilon\mathbb{B}_{\mathcal{E}_2} \right) \subseteq \Gamma(\bar{u}) + \kappa \|u - \bar{u}\|_2 \mathbb{B}_{\mathcal{E}_2} \quad \mbox{for all } u\in\mathcal{E}_1.
$$
\end{enumerate}
\resetspb
%$\Gamma$ is calm at $\bar{x}$ for $\bar{y}$ (or equivalently, $\Gamma^{-1}$ is metrically subregular at $\bar{y}$ for $\bar{x}$). 
\end{prop}
\begin{proof}
It is clear that (b) implies (a).  Hence, suppose that (a) holds.  By~\eqref{eq:calmness}, given any $\bar{v}\in\bar{V}$, there exist constants $\kappa(\bar{v}),\epsilon(\bar{v})>0$ such that
\begin{equation} \label{eq:cpt-calm}
\Gamma(u) \cap \mathbb{B}_{\mathcal{E}_2}(\bar{v},\epsilon(\bar{v})) \subseteq \Gamma(\bar{u}) + \kappa(\bar{v})\|u-\bar{u}\|_2\mathbb{B}_{\mathcal{E}_2} \quad\mbox{for all } u\in\mathcal{E}_1.
\end{equation}
Let $\mathbb{B}_{\mathcal{E}_2}^\circ$ denote the open unit ball around the origin in $\mathcal{E}_2$.  Then, the set $\bigcup_{\bar{v}\in\bar{V}} \left( \bar{v} + \epsilon(\bar{v})\mathbb{B}_{\mathcal{E}_2}^\circ \right)$ forms an open cover of the compact set $\bar{V}$.  Hence, by the Heine-Borel theorem, there exist $N$ points (where $N\ge1$ is finite) $\bar{v}_1,\ldots,\bar{v}_N \in \bar{V}$ such that $\bar{V}\subseteq\bigcup_{i=1}^N \left( \bar{v}_i + \epsilon(\bar{v}_i)\mathbb{B}_{\mathcal{E}_2}^\circ \right)$.  We claim that there exists a constant $\epsilon>0$ such that $\bar{V}+\epsilon\mathbb{B}_{\mathcal{E}_2} \subseteq \bigcup_{i=1}^N \left( \bar{v}_i + \epsilon(\bar{v}_i)\mathbb{B}_{\mathcal{E}_2}^\circ \right)$.  Indeed, suppose that this is not the case.  Then, for $k=1,2,\ldots$, we can find vectors $w_k,v_k \in \mathcal{E}_2$ such that for $k=1,2,\ldots$, 
\begin{eqnarray*}
v_k &\in& \bar{V}, \\
\| v_k - w_k \|_2 &\le& 1/k, \\
\| w_k - \bar{v}_i \|_2 &\ge& \epsilon(\bar{v}_i) \quad\mbox{for } i=1,\ldots,N.
\end{eqnarray*}
Since $\bar{V}$ is compact and $\{v_k\}_{k\ge1} \subseteq \bar{V}$, by passing to a subsequence if necessary, we may assume that $v_k\limto v$ for some $v \in \bar{V}$.  Then, we have
$$ 0 \le \|v-w_k\|_2 \le \|v-v_k\|_2 + \|v_k-w_k\|_2 \limto 0, $$
which shows that $w_k\limto v$.  On the other hand, since $v\in\bar{V}$, there exists an index $i\in\{1,\ldots,N\}$ such that $\|v-\bar{v}_i\|_2 = \delta < \epsilon(\bar{v}_i)$.  This implies that
$$ \epsilon(\bar{v}_i) \le \|w_k-\bar{v}_i\|_2 \le \|w_k-v\|_2 + \|v-\bar{v}_i\|_2 = \|w_k - v\|_2 + \delta $$
for $k=1,2,\ldots,$ which contradicts the fact that $w_k\limto v$.  Thus, the claim is established.

Now, upon setting $\kappa=\max_{i=1,\ldots,N} \kappa(\bar{v}_i)$, we obtain
$$
\Gamma(u) \cap \left( \bar{V}+\epsilon\mathbb{B}_{\mathcal{E}_2} \right) \subseteq \Gamma(u) \cap \bigcup_{i=1}^N \left( \bar{v}_i + \epsilon(\bar{v}_i)\mathbb{B}_{\mathcal{E}_2}^\circ \right) \subseteq \Gamma(\bar{u}) + \kappa \|u - \bar{u}\|_2 \mathbb{B}_{\mathcal{E}_2} \quad \mbox{for all } u\in\mathcal{E}_1,
$$
where the second inclusion is due to~\eqref{eq:cpt-calm} and the fact that $\bar{v}_i + \epsilon(\bar{v}_i)\mathbb{B}_{\mathcal{E}_2}^\circ \subseteq \mathbb{B}_{\mathcal{E}_2}(\bar{v}_i,\epsilon(\bar{v}_i))$ for $i=1,\ldots,N$.  This completes the proof.
\end{proof}

\section{Sufficient Conditions for the Validity of the Error Bound~\eqref{eq:err-bd}}\label{sec:charac-eb}
Following our discussion in Section~\ref{subsec:sva}, we now show that under Assumptions~\ref{ass:smooth-func} and~\ref{ass:level-bd}, the error bound~\eqref{eq:err-bd} is implied by certain calmness property of the multi-function $\Gamma$ given in~\eqref{eq:def-sol-map}.  This is achieved by exploring the relationships between error bounds defined using different test sets and residual functions.  For the sake of convenience, we shall refer to the multi-function $\Gamma$ given in~\eqref{eq:def-sol-map} as the \emph{solution map} associated with Problem~\eqref{eq:str-cvx-prob} in the sequel.

\subsection{Error Bound with Neighborhood-Based Test Set}
To begin, recall that the error bound~\eqref{eq:err-bd} involves the test set $\left\{ x\in\mathcal{E} \mid F(x) \le \zeta, \, \|R(x)\|_2 \le \epsilon \right\}$, where $\zeta\ge v^*$ is arbitrary and $\epsilon>0$ depends on $\zeta$.  The following proposition shows that under Assumptions~\ref{ass:smooth-func} and~\ref{ass:level-bd}, we can replace the test set by a neighborhood of $\mathcal{X}$.  This would facilitate our analysis of the relationship between the error bound~\eqref{eq:err-bd} and the calmness of the solution map $\Gamma$, as the latter is also defined in terms of a neighborhood of $\mathcal{X}$.
\begin{prop}\label{prop:equiv-variant}
Consider the optimization problem~\eqref{eq:str-cvx-prob}.  Under Assumptions~\ref{ass:smooth-func} and~\ref{ass:level-bd}, the error bound~\eqref{eq:err-bd} holds if there exist constants $\kappa,\rho>0$ such that
\begin{equation} \tag{EBN} \label{eq:err-bd-vari}
d(x,\mathcal{X}) \leq \kappa\|R(x)\|_2 \quad\mbox{for all } x \in {\rm dom}(f) \mbox{ with } d(x,\mathcal{X})\leq\rho. 
\end{equation}
\end{prop}
%In this section, we will first utilize the tools introduced in Section~\ref{subsec:sva} to give a sufficient condition for the error bound~\eqref{eq:err-bd} to hold. 
%%%%%%%%%%%%%%%%%%%%%
%Then, to further enhance the usability of our result, we will show that the calmness property just mentioned can be guaranteed by two simple regularity conditions on the optimal solution set $\mathcal{X}$ (Theorem~\ref{thm:two-conditions}).  
%%%%%%%%%%%%%%%%%%%%%
\begin{proof}
To establish the error bound~\eqref{eq:err-bd}, it suffices to show that for any $\zeta\geq v^*$, there exists an $\epsilon>0$ such that
$$ d(x,\mathcal{X}) \le \rho \quad\mbox{for all } x\in\mathcal{E} \mbox{ with } F(x)\leq\zeta, \, \|R(x)\|_2\leq\epsilon. $$
Suppose that this does not hold.  Then, there exist a scalar $\zeta \ge v^*$ and a sequence $\{x_k\}_{k\ge1}$ in $\mathcal{E}$ such that $F(x_k)\leq\zeta$ for $k=1,2,\ldots$ and $\|R(x_k)\|_2\rightarrow 0$, but $d(x_k,\mathcal{X})>\rho$ for $k=1,2,\ldots$.  Since $\left\{ x\in\mathcal{E} \mid F(x)\leq\zeta \right\}$ is compact by Assumption~\ref{ass:level-bd}, by passing to a subsequence if necessary, we may assume that $x_k\limto\bar{x}$ for some $\bar{x}\in\mathcal{E}$.  Using the fact that $\mbox{prox}_P$ is 1-Lipschitz continuous on $\mathcal{E}$ (see, \eg,~\cite[Lemma 2.4]{CW05}) and $\nabla f$ is continuous on $\mbox{dom}(f)$ (Assumption~\ref{ass:smooth-func}(a-i)), we see that $R$ is continuous on $\mbox{dom}(f)$.  This, together with the fact that $\|R(x_k)\|_2\rightarrow 0$, implies that $\|R(\bar{x})\|_2=0$; \ie, $\bar{x}\in\mathcal{X}$.  However, this contradicts the fact that $d(x_k,\mathcal{X})>\rho$ for $k=1,2,\ldots$, and the proof is completed.
\end{proof}

\medskip
Before we proceed, two remarks are in order.  First, the reverse implication in Proposition~\ref{prop:equiv-variant} is also true if, in addition to Assumptions~\ref{ass:smooth-func} and~\ref{ass:level-bd}, the optimal solution set $\mathcal{X}$ of Problem~\eqref{eq:str-cvx-prob} is contained in the relative interior of $\mbox{dom}(F)$.  However, since we will mostly focus on sufficient conditions for the error bound~\eqref{eq:err-bd} to hold, we will not indulge in proving this here.  Second, for those instances of Problem~\eqref{eq:str-cvx-prob} that do not satisfy Assumption~\ref{ass:level-bd}, one or both of the error bounds~\eqref{eq:err-bd} and~\eqref{eq:err-bd-vari} could fail to hold.  The following example demonstrates such possibility.

\medskip
\noindent{\bf Example} \, Let $C=\left\{ (x,y)\in\R^2 \mid x\le0,\,y\ge0 \right\}$.  Define the function $f:(-\infty,1)\times\R\limto\R$ by
$$ f(x,y) = \left\{
\begin{array}{c@{\quad}l}
y\exp((x-1)/y) & \mbox{if } x<1 \mbox{ and }y>0, \\
0 & \mbox{if } x<1 \mbox{ and }y\le0
\end{array}
\right. $$
and take $P:\R^2\rightarrow\{0\}\cup\{+\infty\}$ to be the indicator function of $C$.  Furthermore, let $F:(-\infty,1)\times\R\limto(-\infty,+\infty]$ be given by $F(x,y)=f(x,y)+P(x,y)$.  It can be verified that $f$ is convex and continuously differentiable on $(-\infty,1)\times\R$ with
$$%\begin{equation} \label{eq:bad-exam-grad}
\nabla f(x,y) = \left\{
\begin{array}{l@{\quad}l}
\displaystyle{ \left[ \exp\left(\frac{x-1}{y}\right),  \left( 1-\frac{x-1}{y} \right)\exp\left(\frac{x-1}{y}\right) \right]^T } & \mbox{if } x\in(-\infty,1) \mbox{ and } y>0, \\
\noalign{\medskip}
\bz & \mbox{if } x\in(-\infty,1) \mbox{ and } y\le0.
\end{array}
\right.
$$%\end{equation}
Moreover, we have
$$ 
\left\{ (x,y) \in \R^2 \mid F(x,y) \le \zeta \right\} = 
\left\{ \begin{array}{l@{\quad}l}
\emptyset & \mbox{if } \zeta < 0, \\
\left\{ (x,y) \in \R^2 \mid x \le 0, \, y=0 \right\} & \mbox{if } \zeta=0, \\
\left\{ (x,y) \in \R^2 \mid x \le \min\{0,1+y\ln(\zeta/y)\},\, y\ge0 \right\} & \mbox{if } \zeta > 0,
\end{array}
\right.
$$
which shows that the level sets of $F$ are closed but not bounded.  It follows that $F$ is a closed proper convex function with $\mbox{dom}(F)=C$ and
$$ 0 = v^* = \min_{(x,y)\in\R^2} F(x,y), \quad \mathcal{X} = \left\{ (x,y) \in \R^2: x \le 0, \, y=0 \right\}. $$

Next, we determine the residual map $R$ on $\mbox{dom}(F)$.  Recall that
$$ R(x,y) = \begin{bmatrix} x \\ y \end{bmatrix} - \mbox{prox}_P\left( \begin{bmatrix} x \\ y \end{bmatrix} - \nabla f(x,y) \right) \quad\mbox{for all }(x,y)\in(-\infty,1)\times\R. $$
Since $P$ is the indicator function of $C$, it is easy to see that $\mbox{prox}_P$ is the projection operator onto $C$.  Note that for each $y\ge0$, the function 
$$ x\mapsto y - \left( 1-\frac{x-1}{y} \right)\exp\left(\frac{x-1}{y}\right) $$
is decreasing in $x\in(-\infty,0]$.  Moreover, it can be verified that $y-(1+1/y)\exp(-1/y) \ge 0$ for all $y\ge0$.  It follows that
$$ R(x,y) = \nabla f(x,y) \quad\mbox{for all } (x,y)\in\mbox{dom}(F). $$
In particular, we have $r(x,y)=\|R(x,y)\|_2=\|\nabla f(x,y)\|_2$ for all $(x,y)\in\mbox{dom}(F)$.

Now, observe that for any $\zeta>v^*=0$, if $(\bar{x},\bar{y})\notin\mathcal{X}$ satisfies $F(\bar{x},\bar{y})\leq\zeta$, then $F(x,\bar{y})\leq\zeta$ for any $x \le \bar{x}$.  However, we have $d((x,\bar{y}),\mathcal{X}) = \|\bar{y}\|_2 \neq 0$ for any $x\le0$ and $\lim_{x\rightarrow-\infty} \|R(x,\bar{y})\|_2 = 0$.  It follows that there do not exist constants $\kappa,\epsilon>0$ such that~\eqref{eq:err-bd} holds.  Similarly, for any $\rho>0$, we have
$$ \left\{ (x,y)\in\mbox{dom}(F) \mid d((x,y),\mathcal{X}) \le \rho \right\} = \left\{ (x,y) \in \mbox{dom}(F) \mid 0 \le y \le \rho \right\}. $$
Since $\lim_{x\limto-\infty} \|R(x,y)\|_2 = 0$ for any $y\ge0$, there does not exist a constant $\kappa>0$ such that~\eqref{eq:err-bd-vari} holds. In fact, the same arguments show that the instance in question does not possess a H\"{o}lderian error bound; \ie, the error bounds~\eqref{eq:err-bd} and~\eqref{eq:err-bd-vari} fail to hold even if one replaces the inequality $d(x,\mathcal{X}) \le \kappa\|R(x)\|_2$ by $d(x,\mathcal{X}) \le \kappa\|R(x)\|_2^\alpha$ for any $\alpha\in(0,1)$. \endproof

\subsection{Error Bound with Alternative Residual Function}
%Now, let us turn our attention to the residual function $x\mapsto\|R(x)\|_2$ in the error bound~\eqref{eq:err-bd-vari}.  Although this residual function possesses the fundamental property that $\mathcal{X}=\{x\in\mathcal{E} \mid \|R(x)\|_2=0\}$, it does not explicitly account for the structural properties of Problem~\eqref{eq:str-cvx-prob} under Assumption~\ref{ass:smooth-func}.
As the reader would recall, a motivation for using $x\mapsto\|R(x)\|_2$ as the residual function is that the optimal solution set $\mathcal{X}$ can be characterized as $\mathcal{X}=\{x\in\mathcal{E} \mid R(x)=\bz\}$. Since $\mathcal{X}$ admits the alternative characterization~\eqref{eq:optimal-set}, we can define another residual function $r_{\rm alt}:\mathcal{E}\limto\R_+$ by
$$ r_{\rm alt}(x) := \|\mathcal{A}(x)-\bar{y}\|_2 + d(-\bar{g},\del P(x)) $$
and consider the error bound
\begin{equation} \tag{EBR} \label{eq:err-bd-alt-res}
d(x,\mathcal{X}) \leq \kappa\left( \|\mathcal{A}(x)-\bar{y}\|_2 + d(-\bar{g},\del P(x)) \right) \quad\mbox{for all } x \in \mathcal{E} \mbox{ with } d(x,\mathcal{X})\leq\rho,
\end{equation}
where $\kappa,\rho>0$ are constants and $\bar{y}\in\mathcal{T}$, $\bar{g}\in\mathcal{E}$ are given in Proposition~\ref{prop:opt-invariant}.  Our interest in the error bound~\eqref{eq:err-bd-alt-res} stems from the following result, which reveals that it is closely related to certain calmness property of the solution map $\Gamma$:
\begin{prop} \label{prop:err-bd-calm}
Suppose that Problem~\eqref{eq:str-cvx-prob} satisfies Assumptions~\ref{ass:smooth-func} and~\ref{ass:level-bd}. Let $\bar{y}\in\mathcal{T}$ and $\bar{g}\in\mathcal{E}$ be as in Proposition~\ref{prop:opt-invariant}.  Then, the error bound~\eqref{eq:err-bd-alt-res} holds if and only if the solution map $\Gamma:\mathcal{T}\times\mathcal{E}\rightrightarrows\mathcal{E}$ is calm at $(\bar{y},\bar{g})$ for any $\bar{x}\in\Gamma(\bar{y},\bar{g})$.
\end{prop}
\begin{proof}
Suppose that the error bound~\eqref{eq:err-bd-alt-res} holds.  Let $(y,g)\in\mathcal{T}\times\mathcal{E}$ be arbitrary and suppose that $x\in\Gamma(y,g)\cap(\mathcal{X}+\rho\mathbb{B}_{\mathcal{E}})$. In particular, we have $\mathcal{A}(x)=y$ and $-g\in\del P(x)$. Using the inequality $(a+b)^2\le2(a^2+b^2)$, which is valid for all $a,b\in\R$, we see from~\eqref{eq:err-bd-alt-res} that 
$$ d(x,\mathcal{X}) \le \kappa(\|y-\bar{y}\|_2+\|g-\bar{g}\|_2) \le \sqrt{2}\kappa\|(y,g)-(\bar{y},\bar{g})\|_2. $$
Since $\mathcal{X}=\Gamma(\bar{y},\bar{g})$, this implies that $x\in\Gamma(\bar{y},\bar{g})+\sqrt{2}\kappa\|(y,g)-(\bar{y},\bar{g})\|_2\mathbb{B}_{\mathcal{E}}$.  Hence, $\Gamma$ is calm at $(\bar{y},\bar{g})$ for any $\bar{x}\in\Gamma(\bar{y},\bar{g})$.

Conversely, suppose that $\Gamma$ is calm at $(\bar{y},\bar{g})$ for any $\bar{x}\in\Gamma(\bar{y},\bar{g})$.  Since $\mathcal{X}=\Gamma(\bar{y},\bar{g})$ is compact by Assumption~\ref{ass:level-bd}, Proposition~\ref{prop:ulc-calm} implies the existence of $\kappa,\rho>0$ such that
\begin{equation}\label{eq:upper-lip-sol}
\Gamma(y,g) \cap \left( \mathcal{X} + \rho\mathbb{B}_{\mathcal{E}} \right) \subseteq \mathcal{X} + \kappa\|(y,g)-(\bar{y},\bar{g})\|_2\mathbb{B}_{\mathcal{E}} \quad\mbox{for all } (y,g)\in\mathcal{T}\times\mathcal{E}.
\end{equation}
Now, let $x\in\mathcal{E}$ be such that $d(x,\mathcal{X})\le\rho$ and $\del P(x)\not=\emptyset$.  Using~\eqref{eq:upper-lip-sol} and the inequality $\sqrt{a^2+b^2}\le a+b$, which is valid for all $a,b\ge0$, we have
$$ d(x,\mathcal{X}) \le \kappa\|(\mathcal{A}(x),g)-(\bar{y},\bar{g})\|_2 \le \kappa\left( \|\mathcal{A}(x)-\bar{y}\|_2 + \|g-\bar{g}\|_2 \right) \quad\mbox{for all } -g \in \del P(x). $$
It follows that the error bound~\eqref{eq:err-bd-alt-res} holds.
\end{proof}

\medskip
Since our goal is to link the error bound~\eqref{eq:err-bd} and the calmness of the solution map $\Gamma$, in view of Propositions~\ref{prop:equiv-variant} and~\ref{prop:err-bd-calm}, it remains to understand the relationship between the error bounds~\eqref{eq:err-bd-vari} and~\eqref{eq:err-bd-alt-res}.  Towards that end, we prove the theorem, which constitutes the first main result of this paper:
\begin{thm}\label{thm:eb-upper-lip}
Consider the optimization problem~\eqref{eq:str-cvx-prob}.  Under Assumptions~\ref{ass:smooth-func} and~\ref{ass:level-bd}, the error bound \eqref{eq:err-bd-vari} holds if and only if the error bound~\eqref{eq:err-bd-alt-res} holds.
\end{thm}
%Suppose that Problem~\eqref{eq:str-cvx-prob} satisfies Assumptions~\ref{ass:smooth-func} and~\ref{ass:level-bd}.  Let $\bar{y}\in\mathcal{T}$ and $\bar{g}\in\mathcal{E}$ be as in Proposition~\ref{prop:opt-invariant}.  Then, the error bound~\eqref{eq:err-bd-vari} holds if and only if the solution map $\Gamma:\mathcal{T}\times\mathcal{E}\rightrightarrows\mathcal{E}$ is calm at $(\bar{y},\bar{g})$ for any $\bar{x}\in\Gamma(\bar{y},\bar{g})$.
The proof of Theorem~\ref{thm:eb-upper-lip} relies on the following technical result:
\begin{lemma}\label{lem:residual-map}
Suppose that Problem~\eqref{eq:str-cvx-prob} satisfies Assumptions~\ref{ass:smooth-func} and~\ref{ass:level-bd}.  Then, there exists a constant $\rho_0>0$ such that $N_\rho := \left\{ x \in \mathcal{E} \mid d(x,\mathcal{X}) \le \rho \right\} \subseteq {\rm dom}(f)$ for all $\rho\in(0,\rho_0]$.  Moreover, there exist constants $L_A,L_R>0$, which depend on $\rho_0$, such that for all $x\in N_{\rho_0}$, we have
$$ \|\nabla f(x) - \bar{g}\|_2 \leq L_A\|\mathcal{A}(x) - \bar{y}\|_2 \quad\mbox{and}\quad \| R(x) \|_2 \leq L_R\cdot d(x,\mathcal{X}), $$
%\begin{eqnarray*}
%&& \|\nabla f(x) - \bar{g}\|_2 \leq L_A\|\mathcal{A}(x) - \bar{y}\|_2 \leq L_f\cdot d(x,\mathcal{X})
%\end{eqnarray*}
where $\bar{y}\in\mathcal{T}$ and $\bar{g}\in\mathcal{E}$ are given in Proposition \ref{prop:opt-invariant}.
\end{lemma}
\begin{proof}
Recall from the discussion in Section~\ref{subsec:setup} that $\mbox{dom}(f)$ is open.  On the other hand, the optimal solution set $\mathcal{X}$ is compact by Assumption~\ref{ass:level-bd}.  Since $\mathcal{X} \subseteq \mbox{dom}(f)$, a standard argument shows that $N_{\rho_0} \subseteq \mbox{dom}(f)$ for some $\rho_0>0$.  Moreover, we have $N_\rho \subseteq N_{\rho'}$ whenever $0<\rho<\rho'$.

Clearly, the set $N_{\rho_0}$ is compact.  This implies that $\mathcal{A}(N_{\rho_0}) = \left\{ \mathcal{A}(x) \in \mathcal{T} \mid x \in N_{\rho_0} \right\} \subseteq \mbox{dom}(h)$ is also compact.  By Assumption~\ref{ass:smooth-func}(a-ii), $\nabla h$ is Lipschitz continuous on $\mathcal{A}(N_{\rho_0})$.  Hence, for any $x\in N_{\rho_0}$, there exists an $L_A>0$ such that
$$ \| \nabla f(x) - \bar{g} \|_2 = \left\| \mathcal{A}^*\nabla h(\mathcal{A}(x)) - \mathcal{A}^*\nabla h(\bar{y}) \right\|_2 \le \|\mathcal{A}^*\| \cdot \|\nabla h(\mathcal{A}(x)) - \nabla h(\bar{y}) \|_2 \le L_A\|\mathcal{A}(x) - \bar{y}\|_2. $$

Now, let $\bar{x}$ be the projection of $x$ onto $\mathcal{X}$.  Then, we have $\bar{y}=\mathcal{A}(\bar{x})$ and $\nabla f(\bar{x})=\bar{g}$ by Proposition~\ref{prop:opt-invariant} and $R(\bar{x})=\bz$.  This, together with the 1-Lipschitz continuity of $\mbox{prox}_P$ on $\mathcal{E}$, implies that
\begin{eqnarray*}
\|R(x)\|_2 &=& \|R(x) - R(\bar{x})\|_2 \\
&=& \| \left( \mbox{prox}_P(x-\nabla f(x))-x \right) - \left( \mbox{prox}_P(\bar{x} - \bar{g})-\bar{x} \right) \|_2 \\
&\leq&  \|x - \bar{x}\|_2 + \|\mbox{prox}_P(x-\nabla f(x)) - \mbox{prox}_P(\bar{x} - \bar{g})\|_2 \\
&\leq& \|x - \bar{x}\|_2 + \| (x-\bar{x}) - (\nabla f(x)-\bar{g})\|_2 \\
&\leq& 2\|x - \bar{x}\|_2 + L_A\|\mathcal{A}(x) - \bar{y}\|_2 \\
&\le& L_R \cdot d(x,\mathcal{X}),
\end{eqnarray*}
where $L_R=L_A\|\mathcal{A}\|+2>0$.  This completes the proof.
\end{proof}

\medskip
\noindent{\bf Proof of Theorem~\ref{thm:eb-upper-lip}}\,\, Suppose that the error bound~\eqref{eq:err-bd-vari} holds.  Then, there exist constants $\kappa_0,\rho>0$ such that
\begin{equation} \label{eq:eb-vari}
d(x,\mathcal{X}) \le \kappa_0\|R(x)\|_2 \quad\mbox{for all } x \in {\rm dom}(f) \mbox{ with } d(x,\mathcal{X}) \le \rho.
\end{equation}
By decreasing $\rho$ if necessary, we may assume that $\{x\in\mathcal{E} \mid d(x,\mathcal{X}) \le \rho\} \subseteq \mbox{dom}(f)$, so that Lemma~\ref{lem:residual-map} applies.  Let $x\in\mathcal{E}$ be such that $d(x,\mathcal{X})\le\rho$ and suppose that $-g\in\del P(x)$.  Using the definition of the proximity operator $\mbox{prox}_P$ in~\eqref{eq:prox-oper}, it is straightforward to verify that $x = \mbox{prox}_P(x-g)$.  This, together with~\eqref{eq:eb-vari} and the definition of the residual map $R$ in~\eqref{eq:residual-map-def}, leads to
\begin{eqnarray}
d(x,\mathcal{X}) &\le& \kappa_0 \|\mbox{prox}_P(x-\nabla f(x)) - x\|_2 \nonumber \\
&=& \kappa_0 \|\mbox{prox}_P(x - \nabla f(x)) - \mbox{prox}_P(x - g)\|_2 \nonumber \\
&\leq& \kappa_0 \|\nabla f(x) - g\|_2 \label{eq:eb-calm-1} \\
&\leq& \kappa_0 \left(\|\nabla f(x) - \bar{g}\|_2 + \|\bar{g} - g\|_2 \right) \nonumber \\
&\leq& \kappa_0 \left( L_A\|\mathcal{A}(x) - \bar{y}\|_2 + \|g-\bar{g}\|_2 \right) \label{eq:eb-calm-2} \\
&\leq& \kappa \left( \|\mathcal{A}(x)-\bar{y}\|_2 + \|g-\bar{g}\|_2 \right), \nonumber %\label{eq:eb-calm-3}
\end{eqnarray}
where~\eqref{eq:eb-calm-1} follows from the 1-Lipschitz continuity of $\mbox{prox}_P$ on $\mathcal{E}$,~\eqref{eq:eb-calm-2} follows from Lemma~\ref{lem:residual-map}, and $\kappa=\kappa_0 \cdot \max\{L_A,1\}>0$.  Since $-g\in\del P(x)$ is arbitrary, we conclude that the error bound~\eqref{eq:err-bd-alt-res} holds.

Conversely, suppose that the error bound~\eqref{eq:err-bd-alt-res} holds.  Then, there exist constants $\kappa_0,\rho_0>0$ such that 
\begin{equation} \label{eq:eb-alt-res}
d(x,\mathcal{X}) \leq \kappa_0\left( \|\mathcal{A}(x)-\bar{y}\|_2 + d(-\bar{g},\del P(x)) \right) \quad\mbox{for all } x \in \mathcal{E} \mbox{ with } d(x,\mathcal{X})\leq\rho_0.
\end{equation}
By Lemma~\ref{lem:residual-map}, there exists a constant $\rho_1>0$ such that $\left\{x\in\mathcal{E} \mid d(x,\mathcal{X}) \le \rho_1\right\} \subseteq \mbox{dom}(f)$.  Set $\rho=\min\{\rho_0,\rho_1\}/(L_R+1)>0$ and let $x\in\mathcal{E}$ be such that $d(x,\mathcal{X})\le\rho$, where $L_R>0$ is given in Lemma~\ref{lem:residual-map}.  Using the definition of the proximity operator $\mbox{prox}_P$ in~\eqref{eq:prox-oper} and the residual map $R$ in~\eqref{eq:residual-map-def}, we have $\mathbf{0} \in R(x) + \nabla f(x) + \partial P(x+R(x))$, or equivalently,
\begin{equation}\label{eq:opt-con-residual}
-(\nabla f(x) + R(x)) \in \partial P(x+R(x)). 
\end{equation}
In addition, the property of projection onto $\mathcal{X}$, the triangle inequality, and Lemma~\ref{lem:residual-map} imply
\begin{equation}\label{eq:step-inter-2-}
d(x+R(x),\mathcal{X}) \le d(x,\mathcal{X}) + \|R(x)\|_2 \le (L_R+1)\rho \le \rho_0.
\end{equation}
It follows from~\eqref{eq:eb-alt-res}--\eqref{eq:step-inter-2-} and Lemma~\ref{lem:residual-map} that
\begin{eqnarray*}
d(x+R(x),\mathcal{X}) &\le& \kappa_0\left( \|\mathcal{A}(x+R(x))-\bar{y}\|_2 + \|\nabla f(x)+R(x)-\bar{g}\|_2 \right) \\
&\le& \kappa_0 \left[ \| \mathcal{A}(x)-\bar{y} \|_2 + \|\nabla f(x) - \bar{g}\|_2 + (\|\mathcal{A}\|+1)\|R(x)\|_2 \right] \\
&\le& \kappa_0 \left[ (L_A+1) \| \mathcal{A}(x)-\bar{y} \|_2 + (\|\mathcal{A}\|+1)\|R(x)\|_2 \right].
\end{eqnarray*}
Hence, we obtain
$$ d(x,\mathcal{X}) \le d(x+R(x),\mathcal{X})+\|R(x)\|_2 \le \kappa_1\left( \|\mathcal{A}(x) - \bar{y}\|_2 + \|R(x)\|_2 \right), $$
where $\kappa_1=\max\{\kappa_0(L_A+1),\kappa_0(\|\mathcal{A}\|+1)+1\}>0$.  Since $(a+b)^2\leq 2(a^2+b^2)$ for any $a,b\in\R$, the above inequality yields
\begin{equation}\label{eq:square-key-step}
d(x,\mathcal{X})^2 \leq 2\kappa_1^2 \left( \|\mathcal{A}(x)-\bar{y}\|_2^2 + \|R(x)\|_2^2 \right) \quad\mbox{for all } x \in \mathcal{E} \mbox{ with } d(x,\mathcal{X}) \le \rho.
\end{equation}

Now, let $V=\{\mathcal{A}(x) \in \mathcal{T} \mid d(x,\mathcal{X})\leq\rho\} \subseteq \mbox{dom}(h)$.  Since $\mathcal{X}$ is compact, $V$ is also compact. By Assumption~\ref{ass:smooth-func}(a-ii), the function $h$ is strongly convex on $V$.  Hence, there exists a $\sigma>0$ such that for any $x\in\mathcal{E}$ satisfying $d(x,\mathcal{X})\leq\rho$, we have
\begin{equation}\label{eq:key-step-str-cvx}
\sigma\|\mathcal{A}(x) - \bar{y}\|_2^2 \leq \left\langle \nabla h(\mathcal{A}(x))-\nabla h(\bar{y}), \mathcal{A}(x) - \bar{y} \right\rangle = \langle \nabla f(x) - \bar{g}, x - \bar{x}\rangle,
\end{equation}
where $\bar{x}$ is the projection of $x$ onto $\mathcal{X}$.  Using the convexity of $P$ and the fact that $-\bar{g}\in\del P(\bar{x})$ and $-g(x)\in\del P(x+R(x))$, we have
$$
\langle \bar{g}-g(x), x+R(x)-\bar{x} \rangle \ge 0.
$$
It follows that 
\begin{eqnarray}
\langle \nabla f(x) - \bar{g}, x - \bar{x} \rangle + \|R(x)\|_2^2 &\le& \langle \bar{g} - \nabla f(x) + \bar{x} - x, R(x) \rangle \nonumber\\
&\le& \left(  \|\nabla f(x) - \nabla f(\bar{x})\|_2 + d(x,\mathcal{X}) \right) \|R(x)\|_2 \nonumber\\
&\le& (L_f+1)\cdot d(x,\mathcal{X}) \cdot \|R(x)\|_2 \label{eq:inner-prod-bd}
\end{eqnarray}
for some constant $L_f>0$, where we use the fact that $\nabla f$ is Lipschitz continuous on the compact set $\{x\in\mathcal{E} \mid d(x,\mathcal{X})\le\rho\}$ in the last inequality (see the discussion in Section~\ref{subsec:setup}).  Since $\|R(x)\|_2^2\ge0$ for all $x\in\mathcal{E}$, we conclude from~\eqref{eq:square-key-step}--\eqref{eq:inner-prod-bd} that
$$ d(x,\mathcal{X})^2 \leq \kappa_2 \left(  d(x,\mathcal{X})\cdot \|R(x)\|_2 + \|R(x)\|_2^2 \right), $$
where $\kappa_2 = 2\kappa_1^2\cdot\max\{(L_f+1)/\sigma,1\}$.  Solving the above quadratic inequality yields
$$ d(x,\mathcal{X}) \leq \kappa \|R(x)\|_2 \quad\mbox{for all } x \in \mathcal{E} \mbox{ with } d(x,\mathcal{X}) \le \rho, $$
where $\kappa=\left( \kappa_2+\sqrt{\kappa_2(\kappa_2+4)} \right)/2$.  This completes the proof. \endproof

\medskip
Upon combining Propositions~\ref{prop:equiv-variant}, \ref{prop:err-bd-calm} and Theorem~\ref{thm:eb-upper-lip}, we obtain the following sufficient condition for the error bound~\eqref{eq:err-bd} to hold:
\begin{coro} \label{cor:orig-eb}
Under the setting of Theorem~\ref{thm:eb-upper-lip}, the error bound~\eqref{eq:err-bd} holds if the solution map $\Gamma:\mathcal{T}\times\mathcal{E}\rightrightarrows\mathcal{E}$ is calm at $(\bar{y},\bar{g})$ for any $\bar{x}\in\Gamma(\bar{y},\bar{g})$.
\end{coro}

\subsection{Verifying the Calmness of the Solution Map $\Gamma$}
Corollary~\ref{cor:orig-eb} reduces the problem of establishing the error bound~\eqref{eq:err-bd} for those instances of Problem~\eqref{eq:str-cvx-prob} that satisfy Assumptions~\ref{ass:smooth-func} and~\ref{ass:level-bd} to that of checking certain calmness property of the solution map $\Gamma$.  The upshot of this reduction is that the latter problem can be tackled using a wide array of tools in set-valued analysis.  As an illustration, let us develop a simple sufficient condition for the calmness property stated in Corollary~\ref{cor:orig-eb} to hold.

To motivate our approach, observe that the solution map $\Gamma$ has a separable structure.  Specifically, we have
$$ \Gamma(y,g) = \left\{ x\in\mathcal{E} \mid \mathcal{A}(x)=y, \, -g\in\del P(x) \right\} = \Gamma_f(y) \cap \Gamma_P(g), $$
where $\Gamma_f:\mathcal{T}\rightrightarrows\mathcal{E}$ and $\Gamma_P:\mathcal{E}\rightrightarrows\mathcal{E}$ are multi-functions defined by
\begin{equation} \label{eq:sol-map-split}
\Gamma_f(y) := \left\{ x\in\mathcal{E} \mid \mathcal{A}(x)=y \right\}, \quad \Gamma_P(g) := \left\{ x\in\mathcal{E} \mid -g\in\del P(x) \right\}.
\end{equation}
Intuitively, if $x\in\mathcal{E}\setminus\mathcal{X}$ is close to both $\Gamma_f(\bar{y})$ and $\Gamma_P(\bar{g})$, then it should be close to $\Gamma_f(\bar{y})\cap \Gamma_P(\bar{g})=\mathcal{X}$.  This suggests that it may be possible to estimate $d(x,\mathcal{X})$ by separately estimating $d(x,\Gamma_f(\bar{y}))$ and $d(x,\Gamma_P(\bar{g}))$.  Such idea can be formalized using the notion of \emph{bounded linear regularity} of a collection of closed convex sets.  We begin with the definition.
\begin{defi}\label{def:linear-reg}
(see, \eg,~\cite[Definition 5.6]{bauschke1996projection})
Let $C_1,\ldots,C_N$ be closed convex subsets of $\mathcal{E}$ with a non-empty intersection $C$.  We say that the collection $\{C_1,\ldots,C_N\}$ is \emph{boundedly linearly regular} if for every bounded subset $B$ of $\mathcal{E}$, there exists a constant $\kappa>0$ such that
$$ d(x,C) \leq \kappa \cdot \max_{i=1,\ldots,N} d(x,C_i) \quad\mbox{for all } x\in B. $$
\end{defi}
Naturally, we are interested in the collection $\mathcal{C}=\{\Gamma_f(\bar{y}),\Gamma_P(\bar{g})\}$.  It is obvious that both $\Gamma_f(\bar{y})$ and $\Gamma_P(\bar{g})$ are convex, and that the former is closed.  Using the fact that $P$ is a closed proper convex function (Assumption~\ref{ass:smooth-func}(b)) and~\cite[Theorem 24.4]{rockafellar1970convex}, we see that $\Gamma_P(\bar{g})$ is closed as well.  In addition, we have $\Gamma_f(\bar{y})\cap \Gamma_P(\bar{g})=\mathcal{X}$, which is non-empty by Assumption~\ref{ass:level-bd}.  Thus, the collection $\mathcal{C}$ satisfies the hypothesis of Definition~\ref{def:linear-reg}.  The following result highlights the relevance of the notion of bounded linear regularity in establishing the calmness property stated in Corollary~\ref{cor:orig-eb}.
\begin{thm}\label{thm:two-conditions}
Suppose that Problem~\eqref{eq:str-cvx-prob} satisfies Assumptions~\ref{ass:smooth-func} and~\ref{ass:level-bd}.  Let $\bar{y}\in\mathcal{T}$ and $\bar{g}\in\mathcal{E}$ be as in Proposition~\ref{prop:opt-invariant}.  Consider the collection $\mathcal{C}=\{\Gamma_f(\bar{y}),\Gamma_P(\bar{g})\}$, where the multi-functions $\Gamma_f:\mathcal{T}\rightrightarrows\mathcal{E}$ and $\Gamma_P:\mathcal{E}\rightrightarrows\mathcal{E}$ are defined in~\eqref{eq:sol-map-split}.  Suppose that the following two conditions hold:

\bigskip
\noindent (C1). The collection $\mathcal{C}$ is boundedly linearly regular. 

\medskip
\noindent (C2). For any $\bar{x}\in\mathcal{X}$, the subdifferential $\partial P:\mathcal{E}\rightrightarrows\mathcal{E}$ is metrically sub-regular at $\bar{x}$ for $-\bar{g}$.

\bigskip
\noindent Then, the solution map $\Gamma:\mathcal{T}\times\mathcal{E}\rightrightarrows\mathcal{E}$ is calm at $(\bar{y},\bar{g})$ for any $\bar{x}\in\Gamma(\bar{y},\bar{g})$.
\end{thm}
\begin{proof}
Condition (C2) and Fact~\ref{fact:calm-metricsub} imply that $(\del P)^{-1}:\mathcal{E}\rightrightarrows\mathcal{E}$ is calm at $-\bar{g}$ for any $\bar{x}\in\mathcal{X}$.  Since $\mathcal{X}$ is compact by Assumption~\ref{ass:level-bd}, Proposition~\ref{prop:ulc-calm} implies the existence of constants $\kappa_0,\epsilon>0$ such that 
$$ (\del P)^{-1}(-g) \cap (\mathcal{X}+\epsilon\mathbb{B}_{\mathcal{E}}) \subseteq (\del P)^{-1}(-\bar{g}) + \kappa_0\|g-\bar{g}\|_2\mathbb{B}_{\mathcal{E}} \quad\mbox{for all } g \in \mathcal{E}. $$
It is clear that $\Gamma(y,g) \subseteq \Gamma_P(g) = (\del P)^{-1}(-g)$ for any $(y,g)\in\mathcal{T}\times\mathcal{E}$.  Thus, the above inclusion leads to
$$ \Gamma(y,g) \cap (\mathcal{X}+\epsilon\mathbb{B}_{\mathcal{E}}) \subseteq \Gamma_P(\bar{g}) + \kappa_0\|g-\bar{g}\|_2\mathbb{B}_{\mathcal{E}} \quad\mbox{for all } (y,g) \in \mathcal{T}\times\mathcal{E}. $$
In particular, for any $(y,g) \in \mathcal{T}\times\mathcal{E}$ such that $\Gamma(y,g) \cap (\mathcal{X}+\epsilon\mathbb{B}_{\mathcal{E}})\not=\emptyset$, we have
\begin{equation}\label{eq:dist-x-CP}
d(x,\Gamma_P(\bar{g})) \leq \kappa_0 \|g - \bar{g}\|_2 \quad\mbox{for all } x\in\Gamma(y,g) \cap (\mathcal{X}+\epsilon\mathbb{B}_{\mathcal{E}}).
\end{equation}
On the other hand, observe that $\Gamma_f(\bar{y})$ is the set of solutions to a linear system.  Thus, by the Hoffman bound~\cite{H52}, there exists a constant $\kappa_1>0$ such that
\begin{equation}\label{eq:dist-x-Cf}
d(x,\Gamma_f(\bar{y})) \leq \kappa_1 \|\mathcal{A}(x) - \bar{y}\|_2 \quad\mbox{for all } x\in\mathcal{E}.
\end{equation}
Now, since $\mathcal{X}+\epsilon\mathbb{B}_{\mathcal{E}}$ is bounded, by condition (C1), there exists a constant $\kappa_2>0$ such that
\begin{equation}\label{eq:dist-x-opt}
d(x,\Gamma(\bar{y},\bar{g})) \leq \kappa_2 \cdot \max\left\{ d(x,\Gamma_f(\bar{y})), d(x,\Gamma_P(\bar{g})) \right\} \quad\mbox{for all } x \in \mathcal{X}+\epsilon\mathbb{B}_{\mathcal{E}}.
\end{equation}
It follows from~\eqref{eq:dist-x-CP}--\eqref{eq:dist-x-opt} that for any $(y,g) \in \mathcal{T}\times\mathcal{E}$ satisfying $\Gamma(y,g) \cap (\mathcal{X}+\epsilon\mathbb{B}_{\mathcal{E}})\not=\emptyset$, we have
\begin{eqnarray*}
d(x,\Gamma(\bar{y},\bar{g})) &\le& \kappa_2 \cdot \max\left\{ \kappa_1\|\mathcal{A}(x)-\bar{y}\|_2, \kappa_0\|g-\bar{g}\|_2 \right\} \\
&\le& \kappa\cdot\|(y,g)-(\bar{y},\bar{g})\|_2 \quad\mbox{for all }x\in\Gamma(y,g) \cap (\mathcal{X}+\epsilon\mathbb{B}_{\mathcal{E}}),
\end{eqnarray*}
where $\kappa=\kappa_2\cdot\max\{\kappa_0,\kappa_1\}$.  This implies that $\Gamma$ is calm at $(\bar{y},\bar{g})$ for any $\bar{x}\in\Gamma(\bar{y},\bar{g})$, as desired.
\end{proof}

\medskip
As seen from Theorem~\ref{thm:two-conditions}, the bounded linear regularity of the collection $\mathcal{C}$ can potentially simplify the task of verifying the calmness property stated in Corollary~\ref{cor:orig-eb} and hence of establishing the error bound~\eqref{eq:err-bd}.  Thus, it is natural to ask when the collection $\mathcal{C}$ is boundedly linearly regular.  The following fact provides a simple answer.
\begin{fact}\label{fact:linear-regular}
(\cite[Corollary 3]{bauschke1999strong})
Let $C_1,\ldots, C_N$ be closed convex subsets of $\mathcal{E}$, where $C_{r+1},\ldots,C_N$ are polyhedral for some $r\in\{0,1,\ldots,N\}$.  Suppose that 
$$ \bigcap_{i=1}^r {\rm ri}(C_i) \cap \bigcap_{i=r+1}^N C_i \neq\emptyset. $$
Then, the collection $\{C_1,\ldots,C_N\}$ is boundedly linearly regular.
\end{fact}

Although Fact~\ref{fact:linear-regular} only gives a sufficient condition for the collection $\mathcal{C}$ to be boundedly linearly regular, it is already very useful for studying the error bound property~\eqref{eq:err-bd} associated with Problem~\eqref{eq:str-cvx-prob}.  This will be elaborated in the next section.

\section{Applications to Structured Convex Optimization} \label{sec:sco}
So far our investigation has focused on deriving conditions that can imply the error bound~\eqref{eq:err-bd} for the structured convex optimization problem~\eqref{eq:str-cvx-prob}.  However, we have yet to exhibit instances of Problem~\eqref{eq:str-cvx-prob} that would satisfy those conditions.  As it turns out, such instances abound in applications.  In this section, we will consider four classes of instances of Problem~\eqref{eq:str-cvx-prob} and show that they all possess the calmness property stated in Corollary~\ref{cor:orig-eb}.  Consequently, they all have the error bound property~\eqref{eq:err-bd}.  Although previous works have already established the error bound property for three of the four classes of instances mentioned above, we shall see that our approach provides a unified and more transparent treatment of the existing results.  More interestingly, our approach allows us to resolve the validity of the error bound~\eqref{eq:err-bd} for the fourth class of instances, which comprises of structured convex optimization problems with nuclear norm regularization.  This answers an open question raised by Tseng~\cite{tseng2010approximation}.

For notational simplicity, in what follows, we shall refer to $f$ as the \emph{loss function} and $P$ as the \emph{regularizer}.  Moreover, unless otherwise stated, Assumptions~\ref{ass:smooth-func} and~\ref{ass:level-bd} will be in force.

\subsection{Strongly Convex Loss Function} \label{subsec:str-cvx}
As a warm-up, suppose that the linear operator $\mathcal{A}$ in Assumption~\ref{ass:smooth-func} is the identity; \ie, $\mathcal{E}=\mathcal{T}$ and $\mathcal{A}(x)=x$ for all $x\in\mathcal{E}$.  This gives rise to instances of Problem~\eqref{eq:str-cvx-prob} in which the loss function $f$ is strongly convex and has a Lipschitz continuous gradient on any compact convex set $V\subseteq\mbox{dom}(f)$.  Conversely, any such loss function can be put into the form~\eqref{eq:form-of-f} by letting $\mathcal{A}$ to be the identity map, $h=f$, and $c=\bz$. It is well known that in this case the error bound~\eqref{eq:err-bd} holds whenever the optimal solution set $\mathcal{X}$ is non-empty; see, \eg,~\cite[Theorem 3.1]{pang1987posteriori}. To recover this result using the machinery developed in Section~\ref{sec:charac-eb}, we first observe that the solution map $\Gamma$ is given by
\begin{equation} \label{eq:str-cvx-sol-map}
\Gamma(y,g) = \left\{
\begin{array}{c@{\quad}l}
\{y\} & \mbox{if $-g \in \del P(y)$}, \\
\emptyset & \mbox{otherwise}.
\end{array}
\right.
\end{equation}
Now, note that $\mathcal{X}$ is either empty or a singleton.  Thus, the non-emptiness of $\mathcal{X}$ is equivalent to Assumption~\ref{ass:level-bd}.  In particular, we have $\mathcal{X}=\Gamma(\bar{y},\bar{g})=\{\bar{x}\}$, where $\bar{x}=\bar{y}\in\mathcal{E}$ and $\bar{g}=\nabla h(\bar{x})+c \in \mathcal{E}$ with $-\bar{g}\in\del P(\bar{x})$.  This, together with~\eqref{eq:str-cvx-sol-map}, implies that
$$ \Gamma(y,g) \subseteq \Gamma(\bar{y},\bar{g}) + \|y-\bar{y}\|_2\mathbb{B}_{\mathcal{E}} \subseteq \Gamma(\bar{y},\bar{g}) + \|(y,g)-(\bar{y},\bar{g})\|_2\mathbb{B}_{\mathcal{E}} \quad\mbox{for all } (y,g) \in \mathcal{T}\times\mathcal{E}, $$
which in turn implies that $\Gamma$ is calm at $(\bar{y},\bar{g})$ for $\bar{x}\in\Gamma(\bar{y},\bar{g})$.  The desired conclusion then follows from Corollary~\ref{cor:orig-eb}.
\begin{prop}
Suppose that Problem~\eqref{eq:str-cvx-prob} satisfies Assumptions~\ref{ass:smooth-func} and~\ref{ass:level-bd}.  Suppose further that $\mathcal{A}$ is the identity map, so that the loss function $f$ is strongly convex and has a Lipschitz continuous gradient on any compact convex set $V\subseteq\mbox{dom}(f)$. Then, the error bound~\eqref{eq:err-bd} holds.
\end{prop}

\subsection{Polyhedral Convex Regularizer} \label{subsec:polyhedral}
Next, we consider instances of Problem~\eqref{eq:str-cvx-prob} in which the regularizer $P$ is polyhedral convex.  This covers settings where $P$ is the LASSO regularizer (\ie, $P(x)=\|x\|_1$) or the $\ell_\infty$-norm regularizer (\ie, $P(x)=\|x\|_\infty$).  In view of Corollary~\ref{cor:orig-eb}, to establish the error bound~\eqref{eq:err-bd}, it suffices to check conditions (C1) and (C2) in Theorem~\ref{thm:two-conditions}.  Towards that end, let us write $\Gamma(y,g)=\Gamma_f(y)\cap\Gamma_P(g)$, where $\Gamma_f,\Gamma_P$ are defined in~\eqref{eq:sol-map-split}.  Since $P$ is a polyhedral convex function, its conjugate $P^*$ is also a polyhedral convex function~\cite[Theorem 19.2]{rockafellar1970convex}.  Thus, by~\cite[Corollary 23.5.1 and Theorem 23.10]{rockafellar1970convex}, the set
$$ \del P^*(-\bar{g}) = (\del P)^{-1}(-\bar{g}) = \Gamma_P(\bar{g}) $$
is polyhedral convex.  As $\Gamma_f(\bar{y})$ is also polyhedral convex (it is the set of solutions to a linear system), we conclude from Fact~\ref{fact:linear-regular} that the collection $\{\Gamma_f(\bar{y}),\Gamma_P(\bar{g})\}$ is boundedly linearly regular; \ie, condition (C1) is satisfied.  

Now, by~\cite[Proposition 3]{R76}, $\del P^*=(\del P)^{-1}$ is a \emph{polyhedral multi-function}; \ie, $\mbox{gph}(\del P^*)$ is the union of a finite (possibly empty) collection of polyhedral convex sets (see~\cite{ZZS15} for an alternative proof of this result).  Hence, we can invoke a celebrated result of Robinson~\cite{robinson1981some} to conclude that $(\del P)^{-1}$ is calm at $-\bar{g}$ for any $\bar{x}\in\mathcal{X}$; see~\cite[Proposition 3H.1]{dontchev2009implicit}.  This, together with Fact~\ref{fact:calm-metricsub}, implies that for any $\bar{x}\in\mathcal{X}$, $\del P$ is metrically sub-regular at $\bar{x}$ for $-\bar{g}$; \ie, condition (C2) is satisfied.
\begin{prop}
Suppose that Problem~\eqref{eq:str-cvx-prob} satisfies Assumptions~\ref{ass:smooth-func} and~\ref{ass:level-bd} with $P$ being a polyhedral convex regularizer. Then, the error bound~\eqref{eq:err-bd} holds.
\end{prop}

Modulo the boundedness assumption on $\mathcal{X}$ (see Assumption~\ref{ass:level-bd}), our argument above leads to an alternative proof of Theorem 2.1 in Luo and Tseng~\cite{luo1992linear} and a part of Theorem 4 in Tseng and Yun~\cite{tseng2009coordinate}.  It is worth noting that one can use the machinery developed in Section~\ref{sec:charac-eb} to establish the error bound~\eqref{eq:err-bd} without assuming the boundedness of $\mathcal{X}$.  However, one needs to exploit the polyhedrality of $P$, just as it was done in~\cite{luo1992linear,tseng2009coordinate}. Since our original motivation is to develop an analysis framework that can tackle non-polyhedral regularizers $P$, we choose not to pursue a separate, more refined analysis for the polyhedral case, so as to streamline the presentation.

\subsection{Grouped LASSO Regularizer} \label{subsec:groupLASSO}
Now, let us consider Problem~\eqref{eq:str-cvx-prob} with grouped LASSO regularization; \ie, $\mathcal{E}=\R^n$ and $P:\R^n\limto\R$ takes the form $P(x)=\sum_{J\in\mathcal{J}} \omega_J\|x_J\|_2$, where $\mathcal{J}$ is a partition of $\{1,\ldots,n\}$, $x_J\in\R^{|J|}$ is the vector obtained by restricting $x\in\R^n$ to the entries in $J\in\mathcal{J}$, and $\omega_J\ge0$ is a given parameter. The grouped LASSO regularizer, which is convex but not polyhedral in general, is motivated by the desire to induce sparsity among the variables at a group level and has found many applications in statistics; see, \eg,~\cite{LZ06,yuan2006model,meier2008group}.  In a groundbreaking work, Tseng~\cite{tseng2010approximation} showed that the error bound~\eqref{eq:err-bd} holds in this case.  However, his proof involves a delicate and tedious argument.  In particular, it does not offer much insight into how the grouped LASSO regularizer is different from other non-polyhedral convex regularizers, for which a Lipschitzian error bound similar to~\eqref{eq:err-bd} typically does not hold without further assumptions (we shall see one such example in the next sub-section).  In what follows, we will provide an alternative, more transparent proof of Tseng's result using the machinery developed in Section~\ref{sec:charac-eb}.  The proof reveals that the grouped LASSO regularizer and the solution map $\Gamma$ it induces possess nice structural and regularity properties.  Such properties make it possible to establish the error bound~\eqref{eq:err-bd} even though the grouped LASSO regularizer is non-polyhedral.
%Our proof reveals that the validity of the error bound~\eqref{eq:err-bd} in this case can be attributed to the nice regularity properties of the solution map $\Gamma$ and regularizer $P$, as they satisfy conditions (C1) and (C2) in Theorem~\ref{thm:two-conditions}.

To begin, recall that for any $x\in\mathcal{F}$, where $\mathcal{F}$ is a finite-dimensional Euclidean space, we have
\begin{equation} \label{eq:subdiff-2norm} 
\del\|x\|_2 = \left\{ s \in \mathcal{F} \mid \|s\|_2 \le 1, \, \langle s,x \rangle = \|x\|_2 \right\} = \left\{
\begin{array}{c@{\quad}l}
\mathbb{B}_{\mathcal{F}} & \mbox{if } x=\bz, \\
\noalign{\smallskip}
x/\|x\|_2 & \mbox{otherwise}.
\end{array}
\right.
\end{equation}
Since $\del P(x)=\sum_{J\in\mathcal{J}} \omega_J\del\|x_J\|_2$ by~\cite[Theorem 23.8]{rockafellar1970convex} and $\mathcal{J}$ is a partition of $\{1,\ldots,n\}$, a simple calculation shows that $-g\in\del P(x)$ if and only if $-g_J \in \omega_J\del\|x_J\|_2$ for all $J\in\mathcal{J}$. In particular, for any $g\in\R^n$, we have
\begin{equation} \label{eq:cartesian-prod-group-lasso}
\Gamma_P(g) = \prod_{J\in\mathcal{J}} \Gamma_{P,J}(g),
\end{equation}
where
$$ \Gamma_{P,J}(g) := \left\{ x\in\R^{|J|} \mid -g_J \in \omega_J\del\|x\|_2 \right\} \quad\mbox{for } J \in \mathcal{J}. $$
The following result provides an explicit characterization of $\Gamma_{P,J}(g)$, where $J\in\mathcal{J}$:
\begin{prop} \label{prop:x-in-sub-2norm}
Let $J\in\mathcal{J}$ and $g\in\R^n$ be fixed.  If $\omega_J=0$, then
$$ 
\Gamma_{P,J}(g) = \left\{
\begin{array}{l@{\quad}l}
\R^{|J|} & \mbox{if } g_J = \bz, \\
\noalign{\smallskip}
\emptyset & \mbox{otherwise}.
\end{array}
\right.
$$
On the other hand, if $\omega_J>0$, then
$$ 
\Gamma_{P,J}(g) = \left\{
\begin{array}{l@{\quad}l}
\emptyset & \mbox{if } \|g_J\|_2 > \omega_J, \\
\noalign{\smallskip}
\left\{ a \cdot g_J \in \R^{|J|} \mid a \le 0 \right\} & \mbox{if } \|g_J\|_2=\omega_J, \\
\noalign{\smallskip}
\{\bz\} & \mbox{if } \|g_J\|_2 < \omega_J.
\end{array}
\right.
$$
Consequently, $\Gamma_{P,J}(g)$ is a polyhedral convex set.
\end{prop}
\begin{proof}
The case where $\omega_J=0$ is trivial.  Thus, let us focus on the case where $\omega_J>0$.  Using~\eqref{eq:subdiff-2norm}, it is clear that $\Gamma_{P,J}(g)=\emptyset$ (resp.~$\Gamma_{P,J}(g)=\{\bz\}$) when $\|g_J\|_2>\omega_J$ (resp.~$\|g_J\|_2 < \omega_J$).  Now, suppose that $\|g_J\|_2=\omega_J$ and $x\in\Gamma_{P,J}(g)$.  By~\eqref{eq:subdiff-2norm} and the Cauchy-Schwarz inequality, we have $\|x\|_2 = \langle -g_J/\omega_J,x \rangle \le \|x\|_2$, which implies that $x$ is a non-negative multiple of $-g_J/\omega_J$.  Conversely, it is easy to see that $a\cdot g_J \in \Gamma_{P,J}(g)$ for any $a \le 0$.  This completes the proof.
\end{proof}

\medskip
From~\eqref{eq:cartesian-prod-group-lasso} and Proposition~\ref{prop:x-in-sub-2norm}, we see that $\Gamma_P(\bar{g})$ is a polyhedral convex set. Thus, by Fact~\ref{fact:linear-regular}, the collection $\{\Gamma_f(\bar{y}),\Gamma_P(\bar{g})\}$ is boundedly linearly regular; \ie, condition (C1) in Theorem~\ref{thm:two-conditions} is satisfied.  

Next, we show that for any $\bar{x}\in\mathcal{X}$, $\del P$ is metrically sub-regular at $\bar{x}$ for $-\bar{g}$; \ie, condition (C2) in Theorem~\ref{thm:two-conditions} is also satisfied.  Using the product structure~\eqref{eq:cartesian-prod-group-lasso} of $\Gamma_P$, for any $x\in\R^n$, we have
$$ d\left( x, (\del P)^{-1}(-\bar{g}) \right)^2 = d\left( x,\Gamma_P(\bar{g}) \right)^2 = \sum_{J\in\mathcal{J}} d\left( x_J,\Gamma_{P,J}(\bar{g}) \right)^2 = \sum_{J\in\mathcal{J}} d\left( x_J, (\omega_J\del\|\cdot\|_2)^{-1}(-\bar{g}_J) \right)^2. $$
Moreover, observe that
$$ d\left( -\bar{g},\del P(x) \right)^2 = \sum_{J\in\mathcal{J}} d\left( -\bar{g}_J, \omega_J\del\|x_J\|_2 \right)^2. $$
Since $(\bar{x},-\bar{g})\in{\rm gph}(\del P)$, Proposition~\ref{prop:x-in-sub-2norm} implies that $\bar{g}_J=\bz$ whenever $\omega_J=0$, where $J\in\mathcal{J}$.  This in turn implies that
\begin{eqnarray}
d\left( x, (\del P)^{-1}(-\bar{g}) \right)^2 &=&  \sum_{J\in\mathcal{J} \, \mid \, \omega_J>0} d\left( x_J, (\omega_J\del\|\cdot\|_2)^{-1}(-\bar{g}_J) \right)^2 \nonumber \\
\noalign{\medskip}
&=& \sum_{J\in\mathcal{J} \, \mid \, \omega_J>0} d\left( x_J, (\del\|\cdot\|_2)^{-1}(-\bar{g}_J/\omega_J) \right)^2 \label{eq:2-norm-lhs}
\end{eqnarray}
and
\begin{equation} \label{eq:2-norm-rhs}
d\left( -\bar{g},\del P(x) \right)^2 = \sum_{J\in\mathcal{J} \,\mid\, \omega_J>0} d\left( -\bar{g}_J, \omega_J\del\|x_J\|_2 \right)^2 = \sum_{J\in\mathcal{J} \,\mid\, \omega_J>0} \omega_J^2 \cdot d\left( -\bar{g}_J/\omega_J,\del\|x_J\|_2 \right)^2.
\end{equation}
To proceed, we prove the following result, which would allow us to bound each summand in~\eqref{eq:2-norm-lhs} by the corresponding summand in~\eqref{eq:2-norm-rhs}.
\begin{prop} \label{prop:2-norm-metric-subreg}
The multi-function $\del\|\cdot\|_2:\mathcal{F}\rightrightarrows\mathcal{F}$ is metrically sub-regular at any $x\in\mathcal{F}$ for any $s\in\mathcal{F}$ such that $(x,s)\in{\rm gph}(\del\|\cdot\|_2)$.
\end{prop}
\begin{proof}
Let $(x_0,s_0)\in{\rm gph}(\del\|\cdot\|_2)$ be arbitrary.  By~\eqref{eq:subdiff-2norm}, we have $\|s_0\|_2\le1$.  Consider first the case where $\|s_0\|_2<1$.  We have $x_0=\bz$ and $(\del\|\cdot\|_2)^{-1}(s_0)=\{\bz\}$.  It follows that for any $x\in\mathcal{F}$,
$$ d\left( x,(\del\|\cdot\|_2)^{-1}(s_0) \right) = \|x\|_2. $$
On the other hand, set $\epsilon=\min\{ \|y-s_0\|_2 \mid \|y\|_2 = 1\}$.  Since $\|s_0\|_2<1$, we have $\epsilon>0$.  Moreover, for any $x\in\mathcal{F} \setminus\{\bz\}$, we have $d(s_0,\del\|x\|_2) = \|s_0 - (x/\|x\|_2)\|_2 \ge \epsilon$.  Thus, we obtain
$$ d\left( x,(\del\|\cdot\|_2)^{-1}(s_0) \right) \le d(s_0,\del\|x\|_2) \quad\mbox{for all } x \in \mathbb{B}_{\mathcal{F}}(x_0,\epsilon), $$
as required.

Next, consider the case where $\|s_0\|_2=1$.  Let $\epsilon>0$ be arbitrary and set $\kappa=\|x_0\|_2+\epsilon>0$.  We claim that
$$ d\left( x,(\del\|\cdot\|_2)^{-1}(s_0) \right) \le \kappa \cdot d(s_0,\del\|x\|_2) \quad\mbox{for all } x\in\mathbb{B}_{\mathcal{F}}(x_0,\epsilon). $$
Indeed, let $x\in\mathbb{B}_{\mathcal{F}}(x_0,\epsilon)$ be arbitrary. If $x=\bz$, then $d\left( x,(\del\|\cdot\|_2)^{-1}(s_0) \right)=d(s_0,\del\|x\|_2)=0$, which implies that the above inequality holds trivially.  Otherwise, note that by Proposition~\ref{prop:x-in-sub-2norm}, we have $(\del\|\cdot\|_2)^{-1}(s_0) = \left\{ a\cdot s_0 \in \mathcal{F} \mid a \ge 0 \right\}$.  This, together with~\eqref{eq:subdiff-2norm} and the definition of $\kappa$, yields
$$ d\left( x,(\del\|\cdot\|_2)^{-1}(s_0) \right) \le \| x-\|x\|_2\cdot s_0 \|_2 = \|x\|_2 \cdot \| s_0 - (x/\|x\|_2) \|_2 \le \kappa\cdot d(s_0,\del\|x\|_2), $$
as desired.
\end{proof}

\medskip
From Proposition~\ref{prop:2-norm-metric-subreg} and the fact that $(\bar{x}_J,-\bar{g}_J)\in{\rm gph}(\omega_J\del\|\cdot\|_2)$ for $J\in\mathcal{J}$, we deduce that for each $J\in\mathcal{J}$ with $\omega_J>0$, there exist constants $\kappa_J,\epsilon_J>0$ such that 
\begin{equation} \label{eq:comp-metric-subreg}
 d\left( x_J, (\del\|\cdot\|_2)^{-1}(-\bar{g}_J/\omega_J) \right) \le \kappa_J \cdot d\left( -\bar{g}_J/\omega_J,\del\|x_J\|_2 \right) \quad\mbox{for all } x_J \in \mathbb{B}_{\R^{|J|}}(\bar{x}_J,\epsilon_J). 
\end{equation}
%This is accomplished by proving the following more general result:
It then follows from~\eqref{eq:2-norm-lhs},~\eqref{eq:2-norm-rhs}, and~\eqref{eq:comp-metric-subreg} that
\begin{eqnarray*}
d\left( x, (\del P)^{-1}(-\bar{g}) \right)^2 &\le& \sum_{J\in\mathcal{J} \, \mid \, \omega_J>0} \kappa_J^2 \cdot d\left( -\bar{g}_J/\omega_J,\del\|x_J\|_2 \right)^2 \\
\noalign{\medskip}
&\le& \kappa^2 \cdot d\left( -\bar{g},\del P(x) \right)^2 \quad\mbox{for all } x \in \mathbb{B}_{\R^n}(\bar{x},\epsilon),
\end{eqnarray*}
where
$$ \kappa=\max_{J\in\mathcal{J}\,\mid\,\omega_J>0} \frac{\kappa_J}{\omega_J} \quad\mbox{and}\quad \epsilon=\min_{J\in\mathcal{J}\,\mid\,\omega_J>0} \epsilon_J. $$
In other words, $\del P$ is metrically sub-regular at $\bar{x}$ for $-\bar{g}$.

Finally, by invoking Theorem~\ref{thm:two-conditions} and Corollary~\ref{cor:orig-eb}, we recover the following result of Tseng~\cite[Theorem 2]{tseng2010approximation}:
\begin{prop}
Suppose that Problem~\eqref{eq:str-cvx-prob} satisfies Assumptions~\ref{ass:smooth-func} and~\ref{ass:level-bd} with $P$ being the grouped LASSO regularizer. Then, the error bound~\eqref{eq:err-bd} holds.
\end{prop}

%In a recent work~\cite{ZZS15}, we have extended the above arguments to show that if Problem~\eqref{eq:str-cvx-prob} satisfies Assumptions~\ref{ass:smooth-func} and~\ref{ass:level-bd} and $P$ is the $\ell_{1,p}$-norm regularizer (\ie, $\mathcal{E}=\R^m$ and $P:\R^m\limto\R$ takes the form $P(x)=\sum_{J\in\mathcal{J}} \omega_J\|x_J\|_p$) with $p\in[1,2]\cup\{\infty\}$, then the error bound~\eqref{eq:err-bd} holds.

\subsection{Nuclear Norm Regularizer}
Lastly, we consider instances of Problem~\eqref{eq:str-cvx-prob} in which $P$ is the nuclear norm regularizer; \ie, $\mathcal{E}=\R^{m\times n}$ and $P:\R^{m\times n}\limto\R$ is given by $P(X)=\|X\|_*$.  Without loss of generality, we assume that $m\le n$. The nuclear norm regularizer has been widely used in convex relaxations of low-rank matrix optimization problems; see, \eg,~\cite{AEP08,Gro11,JSZ+13} and the references therein. However, to the best of our knowledge, the question of whether the error bound~\eqref{eq:err-bd} holds in this case remains open; see Footnote 1.  In the sequel, we resolve this question by utilizing the machinery developed in Section~\ref{sec:charac-eb}.  Specifically, we first show that the error bound~\eqref{eq:err-bd} holds in this case under certain strict complementarity-type regularity condition.  Then, we give a concrete example to show that without such condition, the error bound~\eqref{eq:err-bd} could fail to hold.

\subsubsection{Basic Results in Matrix Theory} \label{sssec:matrix-basic}
Let us begin by recalling some basic results in matrix theory. Consider an arbitrary matrix $X\in\R^{m\times n}$ of rank $r\le m$, whose $i$-th largest singular value is denoted by $\sigma_i(X)$ for $i=1,\ldots,m$.  Let
\begin{equation} \label{eq:svd-std}
  X = U \begin{bmatrix} \Sigma(X) & \bz \end{bmatrix} V^T = \begin{bmatrix} U_+ & U_0 \end{bmatrix} \begin{bmatrix} \Sigma_+(X) & \bz \\ \bz & \bz \end{bmatrix} \begin{bmatrix} V_+ & V_0 \end{bmatrix}^T
\end{equation}
be any singular value decomposition (SVD) of $X$, where $U=\begin{bmatrix} U_+ & U_0 \end{bmatrix} \in \Or^m$ is orthogonal with $U_+\in\R^{m\times r}$ and $U_0\in\R^{m\times(m-r)}$; $V=\begin{bmatrix} V_+ & V_0 \end{bmatrix} \in \Or^n$ is orthogonal with $V_+\in\R^{n\times r}$ and $V_0\in\R^{n\times(n-r)}$; $\Sigma(X) = \mbox{Diag}(\sigma_1(X),\ldots,\sigma_m(X)) \in \mathbb{S}^m$ and $\Sigma_+(X) = \mbox{Diag}(\sigma_1(X),\ldots,\sigma_r(X)) \in \mathbb{S}^r$ are diagonal.  The following characterization of the subdifferential of the nuclear norm is well known:
\begin{fact} \label{fact:subdiff-nucnorm}
(see, \eg,~\cite[Example 2]{W92}) We have
$$
  \del\|X\|_* = \left\{ \begin{bmatrix} U_+ & U_0 \end{bmatrix} \begin{bmatrix} I_r & \bz \\ \bz & W \end{bmatrix} \begin{bmatrix} V_+ & V_0 \end{bmatrix}^T \,\Big|\, \|W\| \le 1 \right\},
$$
where $I_r$ is the $r\times r$ identity matrix.
\end{fact}

Although the matrices $\Sigma(X)$ and $\Sigma_+(X)$ are uniquely determined by $X$, there could be multiple pairs of orthogonal matrices $(U,V)$ that decompose $X$ into the form~\eqref{eq:svd-std}. Let
$$ \Xi(X) := \left\{ (U,V)\in\Or^m\times\Or^n \mid X = U \begin{bmatrix} \Sigma(X) & \bz \end{bmatrix} V^T \right\} $$
be the set of all such pairs of orthogonal matrices.  Furthermore, let $\bar{\sigma}_1(X)>\bar{\sigma}_2(X)>\cdots>\bar{\sigma}_s(X)$ be the distinct non-zero singular values of $X$.  Then, we can define the index sets
$$ \mathcal{I}_k := \left\{ i \in \{1,\ldots,m\} \mid \sigma_i(X) = \bar{\sigma}_k(X) \right\} \quad\mbox{for } k=1,\ldots,s. $$
%\begin{eqnarray*}
%&& \mathcal{I}_0 = \left\{ i \in \{1,\ldots,m\} \mid \sigma_i(X)=0 \right\}, \\
%\noalign{\smallskip}
%&& \mathcal{I}_k = \left\{ i \in \{1,\ldots,m\} \mid \sigma_i(X) = \bar{\sigma}_k(X) \right\} \quad\mbox{for } k=1,\ldots,s.
%\end{eqnarray*}
The following result explains the relationship between different SVDs of $X$:
\begin{fact} \label{fact:diff-svd}
(see, \eg,~\cite[Proposition 5]{DST14}) Let $(U_1,V_1),(U_2,V_2)\in\Xi(X)$ be given. Then, there exist orthogonal matrices $Q_k\in\Or^{|\mathcal{I}_k|}$ for $k=1,\ldots,s$, $Q'\in\Or^{m-r}$, and $Q''\in\Or^{n-r}$ such that
$$ U_1^TU_2 = \begin{bmatrix} Q & \bz \\ \bz & Q' \end{bmatrix} \quad\mbox{and}\quad V_1^TV_2 = \begin{bmatrix} Q & \bz \\ \bz & Q'' \end{bmatrix}, $$
where $Q={\rm BlkDiag}(Q_1,\ldots,Q_s)\in\Or^r$.
\end{fact}

To facilitate our study of the local behavior of the solution map $\Gamma$, we will also need the following matrix perturbation results:
\begin{fact} \label{fact:sv-lip}
  (see, \eg,~\cite[Chapter IV, Theorem 4.11]{SS90}) For any $X,Y\in\R^{m\times n}$ and $i\in\{1,\ldots,m\}$,
  $$ |\sigma_i(X)-\sigma_i(Y)| \le \|X-Y\|_F. $$
\end{fact}
\begin{fact} \label{fact:svec-ulc}
  (\cf~\cite[Proposition 7]{DST14}) For any $X\in\R^{m\times n}$, there exist constants $\gamma,\delta>0$ such that if $Y\in\R^{m\times n}$ satisfies $\|X-Y\|_F\le\delta$ and $U_Y\in\Or^m$, $V_Y\in\Or^n$ satisfy $(U_Y,V_Y)\in\Xi(Y)$,
%admits the SVD
%  $$ Y = U_Y \begin{bmatrix} \Sigma(Y) & \bz \end{bmatrix} V_Y^T, $$
%  where $U_Y\in\Or^m$, $V_Y\in\Or^n$ are orthogonal and $\Sigma(Y)\in\Sy^m$ is diagonal, 
then there exists a pair of orthogonal matrices $(U_X,V_X)\in\Xi(X)$ such that
  $$ \| (U_X,V_X)-(U_Y,V_Y) \|_F \le \gamma\|X-Y\|_F. $$
\end{fact}

\subsubsection{Characterization of $\Gamma_P$}
Armed with the results in Section~\ref{sssec:matrix-basic}, our goal now is to derive an explicit expression for $\Gamma_P(G)$, where $G\in\R^{m\times n}$ is arbitrary.  Recall that
$$ \Gamma_P(G) = \left\{ X \in \R^{m\times n} \mid -G \in \del\|X\|_* \right\}. $$
Suppose that $\Gamma_P(G)\not=\emptyset$. Then, we have $\|-G\|\le1$ by Fact~\ref{fact:subdiff-nucnorm}.  This allows us to divide the singular values of $-G$ into the following three groups:
$$
\begin{array}{rcl@{\quad}l}
  \sigma_i(-G) &=& 1 & \mbox{for } i=1,\ldots,\bar{s}, \\
  \noalign{\smallskip}
  \sigma_i(-G) &\in& (0,1) & \mbox{for } i=\bar{s}+1,\ldots,\bar{r}, \\
  \noalign{\smallskip}
  \sigma_i(-G) &=& 0 & \mbox{for } i=\bar{r}+1,\ldots,m,
\end{array}
$$
where $\bar{r}=\mbox{rank}(-G)$. In particular, every SVD of $-G$ can be put into the form
\begin{equation}\label{eq:sig-val-dec-G}
-G = \begin{bmatrix}
\bar{U}_1 & \bar{U}_{(0,1)} & \bar{U}_0
\end{bmatrix}\begin{bmatrix}
I_{\bar{s}} & \mathbf{0} & \mathbf{0} \\
\mathbf{0} & \bar{\Sigma}_{(0,1)} & \mathbf{0} \\
\mathbf{0} & \mathbf{0} & \mathbf{0}
\end{bmatrix}\begin{bmatrix}
\bar{V}_1 & \bar{V}_{(0,1)} & \bar{V}_0
\end{bmatrix}^T,
\end{equation}
where $\begin{bmatrix} \bar{U}_1 & \bar{U}_{(0,1)} & \bar{U}_0 \end{bmatrix} \in \Or^m$ is orthogonal with $\bar{U}_1\in\mathbb{R}^{m\times \bar{s}}$, $\bar{U}_{(0,1)}\in\mathbb{R}^{m\times (\bar{r} - \bar{s})}$, and $\bar{U}_0\in\mathbb{R}^{m\times (m-\bar{r})}$; $\begin{bmatrix} \bar{V}_1 & \bar{V}_{(0,1)} & \bar{V}_0 \end{bmatrix} \in \Or^n$ is orthogonal with $\bar{V}_1\in\mathbb{R}^{n\times \bar{s}}$, $\bar{V}_{(0,1)}\in\mathbb{R}^{n\times (\bar{r} - \bar{s})}$, and $\bar{V}_0\in\mathbb{R}^{n\times (n-\bar{r})}$; $\bar{\Sigma}_{(0,1)} = \mbox{Diag}\left( \sigma_{\bar{s}+1}(-G),\ldots,\sigma_{\bar{r}}(-G) \right) \in \mathbb{S}^{\bar{r}-\bar{s}}$ is diagonal.  Using~\eqref{eq:sig-val-dec-G}, we have the following characterization of $\Gamma_P(G)$:
\begin{prop}\label{prop:x-in-sub}
Suppose that $-G$ admits the SVD~\eqref{eq:sig-val-dec-G}. Then, we have
$$%\begin{equation}\label{eq:x-in-sub}
\Gamma_P(G) = \left\{\begin{bmatrix}
\bar{U}_1 & \bar{U}_{(0,1)} & \bar{U}_0
\end{bmatrix}\begin{bmatrix}
Z & \mathbf{0} & \mathbf{0} \\
\mathbf{0} & \mathbf{0} & \mathbf{0} \\
\mathbf{0} & \mathbf{0} & \mathbf{0}
\end{bmatrix}\begin{bmatrix}
\bar{V}_1 & \bar{V}_{(0,1)} & \bar{V}_0
\end{bmatrix}^T \, \Bigg| \, Z\in\mathbb{S}_+^{\bar{s}} \right\}. 
$$%\end{equation}
\end{prop}
\begin{proof}
Let $Z\in\mathbb{S}_+^{\bar{s}}$ be arbitrary and $Z=Q\Lambda Q^T$ be its spectral decomposition, where $Q\in\Or^{\bar{s}}$ is orthogonal and $\Lambda\in\mathbb{S}^{\bar{s}}$ is diagonal.  Consider the matrix
$$ X = \begin{bmatrix}
\bar{U}_1 & \bar{U}_{(0,1)} & \bar{U}_0
\end{bmatrix}\begin{bmatrix}
Z & \mathbf{0} & \mathbf{0} \\
\mathbf{0} & \mathbf{0} & \mathbf{0} \\
\mathbf{0} & \mathbf{0} & \mathbf{0}
\end{bmatrix}\begin{bmatrix}
\bar{V}_1 & \bar{V}_{(0,1)} & \bar{V}_0
\end{bmatrix}^T. $$
Since $Z\succeq\bz$, the diagonal entries of $\Lambda$ are non-negative.  It follows that
$$ X = \begin{bmatrix}
\bar{U}_1Q & \bar{U}_{(0,1)} & \bar{U}_0
\end{bmatrix}\begin{bmatrix}
\Lambda & \mathbf{0} & \mathbf{0} \\
\mathbf{0} & \mathbf{0} & \mathbf{0} \\
\mathbf{0} & \mathbf{0} & \mathbf{0}
\end{bmatrix}\begin{bmatrix}
\bar{V}_1Q & \bar{V}_{(0,1)} & \bar{V}_0
\end{bmatrix}^T $$ 
is an SVD of $X$.  This, together with Fact~\ref{fact:subdiff-nucnorm} and the fact that $\|\bar{\Sigma}_{(0,1)}\| < 1$, implies 
$$ \begin{bmatrix} \bar{U}_1Q & \bar{U}_{(0,1)} & \bar{U}_0
\end{bmatrix}\begin{bmatrix}
I_{\bar{s}} & \mathbf{0} & \mathbf{0} \\
\mathbf{0} & \bar{\Sigma}_{(0,1)} & \bz \\
\mathbf{0} & \bz & \bz
\end{bmatrix}\begin{bmatrix}
\bar{V}_1Q & \bar{V}_{(0,1)} & \bar{V}_0
\end{bmatrix}^T  \in \del\|X\|_*. $$
Upon observing that $QQ^T=I_{\bar{s}}$ and using~\eqref{eq:sig-val-dec-G}, we conclude that $-G\in\del\|X\|_*$, or equivalently, $X\in\Gamma_P(G)$, as desired.

Conversely, let $X\in\R^{m\times n}$ be such that $X \in \Gamma_P(G)$.  Suppose that $\mbox{rank}(X)=r$ and $X$ admits the SVD~\eqref{eq:svd-std}.  Since $-G\in\del\|X\|_*$, Fact~\ref{fact:subdiff-nucnorm} implies the existence of a matrix $W'\in\R^{(m-r)\times(n-r)}$ with $\|W'\|\le1$ such that
$$ -G = \begin{bmatrix} U_+ & U_0 \end{bmatrix} \begin{bmatrix} I_r & \bz \\ \bz & W' \end{bmatrix} \begin{bmatrix} V_+ & V_0 \end{bmatrix}^T. $$
Note that since $\|W'\|\le1$, we have $\sigma_i(W') \in (0,1]$ for $i=1,\ldots,r'$, where $r'=\mbox{rank}(W')$. Now, let
$$ W' = U' \begin{bmatrix} \Sigma_+(W') & \bz \\ \bz & \bz \end{bmatrix} (V')^T $$
be an SVD of $W'$, where $U'\in\Or^{m-r}$, $V'\in\Or^{n-r}$ are orthogonal and $\Sigma_+(W')=\mbox{Diag}(\sigma_1(W'),\ldots,\sigma_{r'}(W'))\in\Sy^{r'}$ is diagonal.  Then, we have the following alternative SVD of $-G$:
\begin{equation} \label{eq:alt-svd-G}
-G = \begin{bmatrix} U_+ & U_0U' \end{bmatrix} \begin{bmatrix} I_r & \bz & \bz \\ \bz & \Sigma_+(W') & \bz \\ \bz & \bz & \bz \end{bmatrix} \begin{bmatrix} V_+ & V_0V' \end{bmatrix}^T.
\end{equation}
Upon comparing~\eqref{eq:sig-val-dec-G} and~\eqref{eq:alt-svd-G} and noting that $I_{\bar{r}-\bar{s}} \succ \bar{\Sigma}_{(0,1)}$ and $I_{r'} \succeq \Sigma_+(W')$, we have $r\le\bar{s}$ and
\begin{eqnarray*}
\left( \begin{bmatrix} \bar{U}_1 & \bar{U}_{(0,1)} & \bar{U}_0 \end{bmatrix}, \begin{bmatrix} \bar{V}_1 & \bar{V}_{(0,1)} & \bar{V}_0 \end{bmatrix} \right) &\in& \Xi(-G), \\
\noalign{\smallskip}
\left( \begin{bmatrix} U_+ & U_0U' \end{bmatrix}, \begin{bmatrix} V_+ & V_0V' \end{bmatrix} \right) &\in& \Xi(-G).
\end{eqnarray*}
By Fact~\ref{fact:diff-svd}, there exist orthogonal matrices $Q\in\Or^{\bar{s}}$, $Q'\in\Or^{m-\bar{s}}$, and $Q''\in\Or^{n-\bar{s}}$ such that
\begin{eqnarray}
\begin{bmatrix} Q & \bz \\ \bz & Q' \end{bmatrix} &=& \begin{bmatrix} \bar{U}_1 & \bar{U}_{(0,1)} & \bar{U}_0 \end{bmatrix}^T\begin{bmatrix} U_+ & U_0U' \end{bmatrix}, \label{eq:G-svd-1} \\
\noalign{\medskip}
\begin{bmatrix} Q & \bz \\ \bz & Q'' \end{bmatrix} &=&  \begin{bmatrix} \bar{V}_1 & \bar{V}_{(0,1)} & \bar{V}_0 \end{bmatrix}^T\begin{bmatrix} V_+ & V_0V' \end{bmatrix}. \label{eq:G-svd-2}
\end{eqnarray}
In particular, we have
\begin{eqnarray}
&& \begin{bmatrix} \bar{U}_1 & \bar{U}_{(0,1)} & \bar{U}_0 \end{bmatrix}^TX\begin{bmatrix} \bar{V}_1 & \bar{V}_{(0,1)} & \bar{V}_0 \end{bmatrix} \nonumber \\
\noalign{\medskip}
&=& \begin{bmatrix} \bar{U}_1 & \bar{U}_{(0,1)} & \bar{U}_0 \end{bmatrix}^T \begin{bmatrix} U_+ & U_0 \end{bmatrix} \begin{bmatrix} \Sigma_+(X) & \bz \\ \bz & \bz \end{bmatrix} \begin{bmatrix} V_+ & V_0 \end{bmatrix}^T\begin{bmatrix} \bar{V}_1 & \bar{V}_{(0,1)} & \bar{V}_0 \end{bmatrix} \label{eq:X-diag-1} \\
\noalign{\medskip}
&=& \begin{bmatrix} Q & \bz \\ \bz & Q' \end{bmatrix} \begin{bmatrix} \Sigma_+(X) & \bz \\ \bz & \bz \end{bmatrix} \begin{bmatrix} Q & \bz \\ \bz & Q'' \end{bmatrix}^T, \label{eq:X-diag-2}
\end{eqnarray}
where~\eqref{eq:X-diag-1} follows from the SVD of $X$ in~\eqref{eq:svd-std};~\eqref{eq:X-diag-2} follows from~\eqref{eq:G-svd-1},~\eqref{eq:G-svd-2}, and the fact that
\begin{align*}
\begin{bmatrix} U_+ & U_0 \end{bmatrix} \begin{bmatrix} \Sigma_+(X) & \bz \\ \bz & \bz \end{bmatrix} \begin{bmatrix} V_+ & V_0 \end{bmatrix}^T &= \begin{bmatrix} U_+ & U_0U' \end{bmatrix} \begin{bmatrix} I_r & \bz \\ \bz & (U')^T \end{bmatrix} \begin{bmatrix} \Sigma_+(X) & \bz \\ \bz & \bz \end{bmatrix} \begin{bmatrix} I_r & \bz \\ \bz & V' \end{bmatrix} \begin{bmatrix} V_+ & V_0V' \end{bmatrix}^T \\
\noalign{\medskip}
&= \begin{bmatrix} U_+ & U_0U' \end{bmatrix} \begin{bmatrix} \Sigma_+(X) & \bz \\ \bz & \bz \end{bmatrix} \begin{bmatrix} V_+ & V_0V' \end{bmatrix}^T.
\end{align*}
Note that $Q\in\Or^{\bar{s}}$ and $\Sigma_+(X)\in\Sy_{++}^r$.  Hence, we cannot multiply $Q$ and $\Sigma_+(X)$ directly. Nevertheless, since $r\le\bar{s}$, we can still expand~\eqref{eq:X-diag-2} to obtain
$$ X = \begin{bmatrix}
\bar{U}_1 & \bar{U}_{(0,1)} & \bar{U}_0
\end{bmatrix}\begin{bmatrix}
Z & \mathbf{0} & \mathbf{0} \\
\mathbf{0} & \mathbf{0} & \mathbf{0} \\
\mathbf{0} & \mathbf{0} & \mathbf{0}
\end{bmatrix}\begin{bmatrix}
\bar{V}_1 & \bar{V}_{(0,1)} & \bar{V}_0
\end{bmatrix}^T $$
for some $Z\in\Sy_+^{\bar{s}}$. This completes the proof.
\end{proof}

\subsubsection{Metric Sub-Regularity of $\del P$}
Our next task is to show that the subdifferential of the nuclear norm, $\del\|\cdot\|_*$, is metrically sub-regular at any $X\in\R^{m\times n}$ for any $-G\in\R^{m\times n}$ such that $(X,-G)\in\mbox{gph}(\del\|\cdot\|_*)$; \ie, for any $(X_0,-G_0)\in\mbox{gph}(\del\|\cdot\|_*)$, there exist constants $\kappa,\epsilon>0$ such that
\begin{equation} \label{eq:nucnorm-metric-subreg}
  d\left( X,(\del\|\cdot\|_*)^{-1}(-G_0) \right) \le \kappa \cdot d\left( -G_0,\del\|X\|_* \right) \quad\mbox{for all } X \in \mathbb{B}_{\R^{m\times n}}(X_0,\epsilon).
\end{equation}
This result, which could be of independent interest, is crucial to understanding the validity of the error bound~\eqref{eq:err-bd} for Problem~\eqref{eq:str-cvx-prob} when $P$ is the nuclear norm regularizer.  Note that by a standard argument (see, \eg,~\cite[Exercise 3H.4]{dontchev2009implicit}), it suffices to establish the existence of constants $\kappa_0,\epsilon_0,\delta_0>0$ such that
\begin{equation} \label{eq:local-nucnorm-metric-subreg} 
d\left( X,(\del\|\cdot\|_*)^{-1}(-G_0) \right) \le \kappa_0 \cdot d\left( -G_0, \del\|X\|_*\cap\mathbb{B}_{\R^{m\times n}}(-G_0,\delta_0) \right) \quad\mbox{for all } X \in \mathbb{B}_{\R^{m\times n}}(X_0,\epsilon_0). 
\end{equation}
Towards that end, let $(X_0,-G_0)\in\mbox{gph}(\del\|\cdot\|_*)$ and $\epsilon_0>0$ be arbitrary, with $\mbox{rank}(-G_0)=\tilde{r}$.  Since $-G_0\in\del\|X_0\|_*$, Fact~\ref{fact:subdiff-nucnorm} implies the existence of an integer $\tilde{s}\in\{1,\ldots,\tilde{r}\}$ such that $\sigma_i(-G_0)=1$ for $i=1,\ldots,\tilde{s}$; $\sigma_i(-G_0)\in(0,1)$ for $i=\tilde{s}+1,\ldots,\tilde{r}$; $\sigma_i(-G_0)=0$ for $i=\tilde{r}+1,\ldots,m$. Hence, we may express any SVD of $-G_0$ as
$$
-G_0 = \begin{bmatrix}
\tilde{U}_1 & \tilde{U}_{(0,1)} & \tilde{U}_0
\end{bmatrix}\begin{bmatrix}
I_{\tilde{s}} & \mathbf{0} & \mathbf{0} \\
\mathbf{0} & \tilde{\Sigma}_{(0,1)} & \mathbf{0} \\
\mathbf{0} & \mathbf{0} & \mathbf{0}
\end{bmatrix}\begin{bmatrix}
\tilde{V}_1 & \tilde{V}_{(0,1)} & \tilde{V}_0
\end{bmatrix}^T,
$$
where 
$\begin{bmatrix} \tilde{U}_1 & \tilde{U}_{(0,1)} & \tilde{U}_0 \end{bmatrix} \in \Or^m$ is orthogonal with $\tilde{U}_1\in\mathbb{R}^{m\times \tilde{s}}$, $\tilde{U}_{(0,1)}\in\mathbb{R}^{m\times (\tilde{r} - \tilde{s})}$, and $\tilde{U}_0\in\mathbb{R}^{m\times (m-\tilde{r})}$; $\begin{bmatrix} \tilde{V}_1 & \tilde{V}_{(0,1)} & \tilde{V}_0 \end{bmatrix} \in \Or^n$ is orthogonal with $\tilde{V}_1\in\mathbb{R}^{n\times \tilde{s}}$, $\tilde{V}_{(0,1)}\in\mathbb{R}^{n\times (\tilde{r} - \tilde{s})}$, and $\tilde{V}_0\in\mathbb{R}^{n\times (n-\tilde{r})}$; $\tilde{\Sigma}_{(0,1)} = \mbox{Diag}\left( \sigma_{\tilde{s}+1}(-G),\ldots,\sigma_{\tilde{r}}(-G) \right) \in \mathbb{S}^{\tilde{r}-\tilde{s}}$ is diagonal.  Now, let $\gamma_0,\delta_0>0$ be the constants that guarantee the property stated in Fact~\ref{fact:svec-ulc} holds at $-G_0$.  Consider a matrix $X\in\R^{m\times n}$ with $\|X-X_0\|_F\le\epsilon_0$. The inequality~\eqref{eq:local-nucnorm-metric-subreg} trivially holds if $\del\|X\|_*\cap\mathbb{B}_{\R^{m\times n}}(-G_0,\delta_0)=\emptyset$.  Hence, suppose that there exists a $-G\in\del\|X\|_*\cap\mathbb{B}_{\R^{m\times n}}(-G_0,\delta_0)$, whose SVD is given by~\eqref{eq:sig-val-dec-G}. In particular, we have $\sigma_i(-G)=1$ for $i=1,\ldots,\bar{s}$; $\sigma_i(-G)\in(0,1)$ for $i=\bar{s}+1,\ldots,\bar{r}$; $\sigma_i(-G)=0$ for $i=\bar{r}+1,\ldots,m$, where $\bar{r}=\mbox{rank}(-G)$. In view of Fact~\ref{fact:sv-lip}, we may, by adjusting $\delta_0>0$ if necessary, assume that $\sigma_i(-G)\in(0,1)$ for $i=\tilde{s}+1,\ldots,\tilde{r}$.  Consequently, we have $\bar{s}\le\tilde{s}$ and $\bar{r}\ge\tilde{r}$.

To proceed, let
\begin{eqnarray*}
&& \bar{U} = \begin{bmatrix} \bar{U}_1 & \bar{U}_{(0,1)} & \bar{U}_0 \end{bmatrix}, \quad\bar{V} = \begin{bmatrix} \bar{V}_1 & \bar{V}_{(0,1)} & \bar{V}_0 \end{bmatrix}^T, \\
\noalign{\medskip}
&& \tilde{U} = \begin{bmatrix} \tilde{U}_1 & \tilde{U}_{(0,1)} & \tilde{U}_0 \end{bmatrix}, \quad \tilde{V} = \begin{bmatrix} \tilde{V}_1 & \tilde{V}_{(0,1)} & \tilde{V}_0 \end{bmatrix}^T.
\end{eqnarray*}
Since $X\in\Gamma_P(G)$ and $X_0\in\Gamma_P(G_0)$, Proposition~\ref{prop:x-in-sub} implies the existence of positive semidefinite matrices $\bar{Z}\in\Sy_+^{\bar{s}}$ and $\tilde{Z}\in\Sy_+^{\tilde{s}}$ such that
\begin{equation} \label{eq:psd-decomp}
X = \bar{U} \begin{bmatrix} \bar{Z} & \mathbf{0} & \mathbf{0} \\ \mathbf{0} & \mathbf{0} & \mathbf{0} \\ \mathbf{0} & \mathbf{0} & \mathbf{0} \end{bmatrix} \bar{V}^T \quad\mbox{and}\quad
X_0 = \tilde{U} \begin{bmatrix} \tilde{Z} & \mathbf{0} & \mathbf{0} \\ \mathbf{0} & \mathbf{0} & \mathbf{0} \\ \mathbf{0} & \mathbf{0} & \mathbf{0} \end{bmatrix} \tilde{V}^T.
\end{equation}
Moreover, since $\|G-G_0\|_F \le \delta_0$, by Fact~\ref{fact:svec-ulc}, there exists a pair $(\tilde{U}^\star,\tilde{V}^\star)\in\Xi(-G_0)$ such that
\begin{equation} \label{eq:nearby-ortho}
  \| (\tilde{U}^\star,\tilde{V}^\star) - (\bar{U},\bar{V}) \|_F \le \gamma_0 \|G-G_0\|_F.
\end{equation}
Upon recall that $(\tilde{U},\tilde{V})\in\Xi(-G_0)$ and invoking Fact~\ref{fact:diff-svd}, we obtain
$$ \tilde{U}^T\tilde{U}^\star = \begin{bmatrix} Q & \bz \\ \bz & Q' \end{bmatrix} \quad\mbox{and}\quad \tilde{V}^T\tilde{V}^\star = \begin{bmatrix} Q & \bz \\ \bz & Q'' \end{bmatrix} $$
for some orthogonal matrices $Q\in\Or^{\tilde{s}}$, $Q'\in\Or^{m-\tilde{s}}$, and $Q''\in\Or^{n-\tilde{s}}$.  Now, consider the matrix
$$ X_0^\star = \tilde{U}^\star \begin{bmatrix} \bar{Z} & \bz & \bz \\ \bz & \bz & \bz \\ \bz & \bz & \bz \end{bmatrix} ( \tilde{V}^\star )^T \in \R^{m\times n}. $$
Observe that since $\bar{s}\le\tilde{s}$, we have
\begin{eqnarray*}
\tilde{U}^TX_0^\star\tilde{V} &=& \tilde{U}^T \tilde{U}^\star \begin{bmatrix} \bar{Z} & \bz \\ \bz & \bz \end{bmatrix} ( \tilde{V}^\star )^T \tilde{V} \\
\noalign{\medskip}
&=& \begin{bmatrix} Q & \bz \\ \bz & Q' \end{bmatrix} \begin{bmatrix} \bar{Z} & \bz \\ \bz & \bz \end{bmatrix} \begin{bmatrix} Q & \bz \\ \bz & Q'' \end{bmatrix}^T \\
\noalign{\medskip}
&=& \begin{bmatrix} \bar{Z}^\star & \bz & \bz \\ \bz & \bz & \bz \\ \bz & \bz & \bz \end{bmatrix}
\end{eqnarray*}
for some $\bar{Z}^\star\in\Sy_+^{\tilde{s}}$.  This, together with Proposition~\ref{prop:x-in-sub}, implies that $X_0^\star\in\Gamma_P(G_0)$. Hence, by setting $\kappa_0=\sqrt{2} \gamma_0 \left( \|X_0\|_F + \epsilon_0 \right)$, we have
\begin{eqnarray}
d\left( X,(\del\|\cdot\|_*)^{-1}(-G_0) \right) &=& d(X,\Gamma_P(G_0)) \nonumber \\
\noalign{\medskip}
&\le& \|X-X_0^\star\|_F \nonumber \\
\noalign{\medskip}
&=& \left\| \bar{U} \begin{bmatrix} \bar{Z} & \bz & \bz \\ \bz & \bz & \bz \\ \bz & \bz & \bz \end{bmatrix} \bar{V}^T - \tilde{U}^\star \begin{bmatrix} \bar{Z} & \bz & \bz \\ \bz & \bz & \bz \\ \bz & \bz & \bz \end{bmatrix} (\tilde{V}^\star)^T \right\|_F \nonumber \\
\noalign{\medskip}
&\le& \sqrt{2} \cdot \|X\|_F \cdot \| (\tilde{U}^\star,\tilde{V}^\star) - (\bar{U},\bar{V}) \|_F \label{eq:dist-subdiff-1} \\
\noalign{\medskip}
&\le& \kappa_0 \cdot \|G-G_0\|_F, \label{eq:dist-subdiff-2}
\end{eqnarray}
where~\eqref{eq:dist-subdiff-1} follows from~\eqref{eq:psd-decomp} and the fact that for any matrix $A\in\R^{m\times n}$ and orthogonal matrices $U_1,U_2\in\Or^m$, $V_1,V_2\in\Or^n$,
\begin{eqnarray*}
  \|U_1AV_1^T-U_2AV_2^T\|_F &=& \| U_1AV_1^T - U_1AV_2^T + U_1AV_2^T - U_2AV_2^T \|_F \\
  &\le& \|A(V_1-V_2)^T\|_F + \|(U_1-U_2)A\|_F \\
  &\le& \sqrt{2} \cdot \|A\|_F \cdot \|(U_1,V_1)-(U_2,V_2)\|_F;
\end{eqnarray*}
\eqref{eq:dist-subdiff-2} follows from~\eqref{eq:nearby-ortho} and the fact that $\|X-X_0\|_F\le\epsilon_0$. Since the above inequality holds for arbitrary $X\in\mathbb{B}_{\R^{m\times n}}(X_0,\epsilon_0)$ and $-G\in\del\|X\|_*\cap\mathbb{B}_{\R^{m\times n}}(-G_0,\delta_0)$, we conclude that~\eqref{eq:local-nucnorm-metric-subreg} holds. Thus, we obtain the following proposition, which constitutes the second main result of this paper:
\begin{prop} \label{prop:nucnorm-metric-subreg}
The multi-function $\del\|\cdot\|_*:\R^{m\times n}\rightrightarrows\R^{m\times n}$ is metrically sub-regular at any $X\in\R^{m\times n}$ for any $G\in\R^{m\times n}$ such that $(X,G)\in{\rm gph}(\del\|\cdot\|_*)$.
\end{prop}

\subsubsection{Validity of the Error Bound~\eqref{eq:err-bd}} \label{subsubsec:nucnorm-eb}
Theorem~\ref{thm:two-conditions} and Proposition~\ref{prop:nucnorm-metric-subreg} imply that in order to establish the error bound~\eqref{eq:err-bd} for the nuclear norm-regularized problem~\eqref{eq:str-cvx-prob}, it suffices to show that the collection $\mathcal{C}=\{\Gamma_f(\bar{y}),\Gamma_P(\bar{G})\}$, where $\bar{y}=\mathcal{A}(X)$ and $\bar{G}=\nabla f(X)$ for any $X\in\mathcal{X}$ (recall Proposition~\ref{prop:opt-invariant}), is boundedly linearly regular. Since Proposition~\ref{prop:x-in-sub} suggests that the set $\Gamma_P(\bar{G})$ is not polyhedral in general, we can invoke Fact~\ref{fact:linear-regular} to conclude that the collection $\mathcal{C}$ is boundedly linearly regular if the regularity condition $\Gamma_f(\bar{y}) \cap \mbox{ri}(\Gamma_P(\bar{G})) \not= \emptyset$ holds. However, such condition is not entirely satisfactory, as it reveals very little about the structure of the optimal solution set $\mathcal{X}$. This motivates us to develop an alternative regularity condition, which leads to the third main result of this paper:

%In view of Theorem~\ref{thm:two-conditions} and Proposition~\ref{prop:nucnorm-metric-subreg}, the validity of the error bound~\eqref{eq:err-bd} for the nuclear norm-regularized problem~\eqref{eq:str-cvx-prob} would follow from the bounded linear regularity of the collection $\mathcal{C}=\{\Gamma_f(\bar{y}),\Gamma_P(\bar{G})\}$, where $\bar{y}=\mathcal{A}(X)$ and $\bar{G}=\nabla f(X)$ for any $X\in\mathcal{X}$ (recall Proposition~\ref{prop:opt-invariant}). As is evident from Proposition~\ref{prop:x-in-sub}, the set $\Gamma_P(\bar{G})$ is not necessarily polyhedral. Thus, certain regularity condition is needed in general to guarantee the bounded linear regularity of $\mathcal{C}$. % Although we know from Fact~\ref{fact:linear-regular} that one such condition is $\Gamma_f(\bar{y}) \cap \mbox{ri}(\Gamma_P(\bar{G})) \not= \emptyset$, it reveals very little about the structure of the optimal solution set $\mathcal{X}$. This motivates us to consider the following alternative condition:

\begin{prop}\label{prop:nucnorm-eb}
Suppose that Problem~\eqref{eq:str-cvx-prob} satisfies Assumptions~\ref{ass:smooth-func} and~\ref{ass:level-bd} with $P$ being the nuclear norm regularizer.  Suppose further that there exists an $X^\star\in\mathcal{X}$ satisfying
\begin{equation}\label{eq:nucnorm-cq}
\bz \in \nabla f(X^{\star}) + {\rm ri}(\del \|X^{\star}\|_*).
\end{equation}
Then, the error bound~\eqref{eq:err-bd} holds.
\end{prop}
\begin{proof}
Recall from~\eqref{eq:optimal-set} that $-\bar{G}\in\del\|X\|_*$ for any $X\in\mathcal{X}$.  Hence, we have $\|-\bar{G}\|\le1$ by Fact~\ref{fact:subdiff-nucnorm}. In particular, we may assume that $-\bar{G}$ admits the SVD~\eqref{eq:sig-val-dec-G}.  Since $X^{\star}\in\mathcal{X}$ satisfies~\eqref{eq:nucnorm-cq}, we have $-\nabla f(X^\star) = -\bar{G} \in \mbox{ri}(\del\|X^{\star}\|_*)$. This, together with~\eqref{eq:sig-val-dec-G} and Fact~\ref{fact:subdiff-nucnorm}, implies that $\mbox{rank}(X^\star)=\bar{s}$. Now, observe that $X^\star\in\Gamma_P(\bar{G})$, as $-\bar{G}\in\del\|X^\star\|_*$. Since $\mbox{rank}(X^\star)=\bar{s}$, Proposition~\ref{prop:x-in-sub} yields $X^\star\in\mbox{ri}(\Gamma_P(\bar{G}))$. Since we also have $X^\star\in\Gamma_f(\bar{y})$, we conclude that $\Gamma_f(\bar{y}) \cap \mbox{ri}(\Gamma_P(\bar{G})) \not=\emptyset$. Hence, by Fact \ref{fact:linear-regular}, the collection $\{\Gamma_f(\bar{y}),\Gamma_P(\bar{G})\}$ is boundedly linearly regular. Upon combining this with Proposition~\ref{prop:nucnorm-metric-subreg} and then invoking Theorem~\ref{thm:two-conditions}, the desired result follows.
\end{proof}

\medskip
To put Proposition~\ref{prop:nucnorm-eb} into perspective, let us make the following remarks:
\begin{enumerate}
\item[\subpb] It is helpful to think of~\eqref{eq:nucnorm-cq} as a strict complementarity condition.  Indeed, suppose that $-\bar{G}=-\nabla f(X^\star)$ admits the SVD~\eqref{eq:sig-val-dec-G} and let $\bar{I}=\begin{bmatrix} \bar{U}_1 & \bar{U}_{(0,1)} & \bar{U}_0 \end{bmatrix}\begin{bmatrix} I_m & \bz \end{bmatrix} \begin{bmatrix} \bar{V}_1 & \bar{V}_{(0,1)} & \bar{V}_0 \end{bmatrix}^T \in \R^{m\times n}$. Then, 
$$%\begin{equation} \label{eq:Igrad}
\bar{I}+\nabla f(X^\star) = 
\begin{bmatrix} \bar{U}_1 & \bar{U}_{(0,1)} & \bar{U}_0 \end{bmatrix}
\begin{bmatrix} 
\bz & \bz & \bz & \bz \\
\bz & I_{\bar{r}-\bar{s}} - \bar{\Sigma}_{(0,1)} & \bz & \bz \\
\bz & \bz & I_{m-\bar{r}} & \bz
\end{bmatrix}
\begin{bmatrix} \bar{V}_1 & \bar{V}_{(0,1)} & \bar{V}_0 \end{bmatrix}^T
$$%\end{equation}
is an SVD of $\bar{I}+\nabla f(X^\star)$.  By Proposition~\ref{prop:x-in-sub}, we may write
\begin{equation} \label{eq:Xstar}
X^\star = \begin{bmatrix}
\bar{U}_1 & \bar{U}_{(0,1)} & \bar{U}_0
\end{bmatrix}\begin{bmatrix}
Z^\star & \mathbf{0} & \mathbf{0} \\
\mathbf{0} & \mathbf{0} & \mathbf{0} \\
\mathbf{0} & \mathbf{0} & \mathbf{0}
\end{bmatrix}\begin{bmatrix}
\bar{V}_1 & \bar{V}_{(0,1)} & \bar{V}_0
\end{bmatrix}^T 
\end{equation}
for some $Z^\star\in\Sy_+^{\bar{s}}$. It follows that $X^\star$ and $\bar{I}+\nabla f(X^\star)$ are complementary in the sense that $\langle X^\star, \bar{I}+\nabla f(X^\star) \rangle=0$.

Now, if the regularity condition~\eqref{eq:nucnorm-cq} holds, then the proof of Proposition~\ref{prop:nucnorm-eb} shows that $X^\star\in\mbox{ri}(\Gamma_P(\bar{G}))$.  In particular, we have $Z^\star\in\Sy_{++}^{\bar{s}}$ in~\eqref{eq:Xstar}. This yields ${\rm rank}( X^\star )+{\rm rank}( \bar{I}+\nabla f(X^\star) )=m$; \ie, $X^\star$ and $\bar{I}+\nabla f(X^\star)$ are strictly complementary.  Conversely, the strict complementarity between $X^\star$ and $\bar{I}+\nabla f(X^\star)$ implies that $Z^\star\in\Sy_{++}^{\bar{s}}$ in~\eqref{eq:Xstar}. Hence, by~\eqref{eq:sig-val-dec-G} and Fact~\ref{fact:subdiff-nucnorm}, we have $-\bar{G}=-\nabla f(X^\star)\in{\rm ri}(\del\|X^\star\|_*)$; \ie, the regularity condition~\eqref{eq:nucnorm-cq} holds.
%$$ 
%\left\langle X^\star(X^\star)^T, \left(\bar{I}+\nabla f(X^\star)\right)\left(\bar{I}+\nabla f(X^\star)\right)^T \right\rangle=0 $$
%and 
%$$ {\rm rank}\left( X^\star(X^\star)^T \right)+{\rm rank}\left( \left(\bar{I}+\nabla f(X^\star)\right) \left(\bar{I}+\nabla f(X^\star)\right)^T \right)=m. $$
%Hence, the matrices $X^\star(X^\star)^T$ and $\left(\bar{I}+\nabla f(X^\star)\right)\left(\bar{I}+\nabla f(X^\star)\right)^T$ are strictly complementary.

\item[\subpb] Using the fact that $\mathcal{X}=\Gamma_f(\bar{y})\cap\Gamma_P(\bar{G})$ and Proposition~\ref{prop:x-in-sub}, we see that $\mathcal{X}$ is the solution set of a linear matrix inequality. As such, it is natural to ask whether the error bound~\eqref{eq:err-bd} in this case follows from existing error bounds for general linear matrix inequalities (see, \eg,~\cite{sturm2000error,AH02}). It turns out that if the decision variable $X$ is a symmetric matrix (\ie, $\mathcal{E}=\mathbb{S}^n$), then it is indeed possible to use the machinery in~\cite{sturm2000error} to establish the error bound~\eqref{eq:err-bd} for Problem~\eqref{eq:str-cvx-prob} under the same regularity condition~\eqref{eq:nucnorm-cq}.  However, the argument is tedious and does not reveal much insight.  Moreover, it is not easy to generalize the argument to handle the case where $X$ is not symmetric or is rectangular. Therefore, we do not pursue such an approach here.
%they are typically concerned with residual functions that are related to the extremal eigenvalue of a certain matrix. As such, it is not clear whether the results in those studies are compatible with ours. An interesting direction for future study would be to investigate the relationship between the error bounds in~\cite{sturm2000error,AH02} and those considered in this paper.
%
%\item[\subpb] In an unpublished manuscript~\cite{PL15}, Pan and Liu show that an error bound similar to~\eqref{eq:err-bd} holds for the nuclear norm-regularized problem~\eqref{eq:str-cvx-prob}, provided that there exists an $X^\star\in\mathcal{X}$ satisfying
\end{enumerate}
\resetspb

In view of Proposition~\ref{prop:nucnorm-eb}, it is natural to ask whether the error bound~\eqref{eq:err-bd} holds without the regularity condition~\eqref{eq:nucnorm-cq}. Unfortunately, the answer is negative in general. To see this, consider the nuclear norm-regularized problem
%In particular, we will construct an example where the error bound of \eqref{eq:nuc-reg} fails because of the lack of \eqref{eq:cons-qualification-1}. 
\begin{equation}\label{eq:counter-ex}
\min_{X\in\R^{2\times 2}} f(X) + \|X\|_{*},
\end{equation}
where $f:\R^{2\times 2}\rightarrow\R$ is the function given by $f(X)=h(\mathcal{A}(X))$, $h:\R^2\rightarrow\R$ is the function given by
$$ 
h(y) = \frac{1}{2} \left\|B^{1/2}y - B^{-1/2}d\right\|_2^2 \quad \mbox{with } B = \begin{bmatrix}
3/2 & -2 \\ -2 & 3 
\end{bmatrix} \succ \bz \mbox{ and } 
d = \begin{bmatrix}
5/2 \\ -1
\end{bmatrix},
$$
and $\mathcal{A}:\R^{2\times 2}\rightarrow\mathbb{R}^2$ is the linear operator given by
$$ \mathcal{A}(X) = (X_{11}, X_{22}). $$
Note that $h$ is continuously differentiable and strongly convex on $\R^2$.  Moreover, we have $\nabla h(y) = By-d$ for any $y\in\R^2$, which implies that $\nabla h$ is Lipschitz continuous on $\R^2$. Thus, Problem~\eqref{eq:counter-ex} satisfies Assumption~\ref{ass:smooth-func}.

Since the adjoint operator of $\mathcal{A}$, denoted by $\mathcal{A}^{*}:\mathbb{R}^2\rightarrow\R^{2\times 2}$, is given by
$$ \mathcal{A}^{*}(y_1,y_2) = 
\begin{bmatrix}
y_1 & 0 \\
0 & y_2
\end{bmatrix},
$$
a straightforward calculation shows that for any $X\in\R^{2\times 2}$,
\begin{equation} \label{eq:gradient-f}
\nabla f(X) = \mathcal{A}^{*}\nabla h(\mathcal{A}(X)) = 
\begin{bmatrix}
\displaystyle{ \frac{3}{2}X_{11}-2X_{22} - \frac{5}{2} } & 0 \\
\noalign{\smallskip}
0 & -2X_{11} + 3X_{22} + 1
\end{bmatrix}.
\end{equation}
Now, consider the matrix
$$ \bar{X} = \begin{bmatrix}
1 & 0 \\ 0 & 0
\end{bmatrix}. $$
Using~\eqref{eq:gradient-f}, we have
$$ \nabla f(\bar{X}) = \begin{bmatrix}
-1 & 0 \\
0 & -1 
\end{bmatrix}.
$$
Moreover, using Fact~\ref{fact:subdiff-nucnorm}, it is easy to verify that
$$ \del\|\bar{X}\|_{*} = \left\{ Z\in\R^{2\times 2} \mid Z_{11} = 1, \; Z_{12} = Z_{21} = 0, \; Z_{22}\in [-1,1] \right\}. $$
Hence, we obtain $-\nabla f(\bar{X}) \in \partial\|\bar{X}\|_{*}$, which shows that $\bar{X}$ is an optimal solution to Problem~\eqref{eq:counter-ex}; \ie, $\bar{X}\in\mathcal{X}$. 

Next, we claim that $\mathcal{X}=\{\bar{X}\}$. Indeed, consider an arbitrary $\tilde{X}\in\mathcal{X}$. Since $h$ is strongly convex on $\R^2$, by Proposition~\ref{prop:opt-invariant}, we have
$$ \mathcal{A}(\tilde{X})=\mathcal{A}(\bar{X})=\begin{bmatrix} 1 \\ 0 \end{bmatrix} \quad\mbox{and}\quad \nabla f(\tilde{X})=\nabla f(\bar{X}) = \begin{bmatrix}
-1 & 0 \\
0 & -1 
\end{bmatrix}. $$
The first relation gives $\tilde{X}_{11}=1$ and $\tilde{X}_{22}=0$. On the other hand, the second relation and the optimality of $\tilde{X}$ imply
$$%\begin{equation}\label{eq:Xtilde}
-\nabla f(\tilde{X}) = \begin{bmatrix}
1 & 0 \\ 0 & 1
\end{bmatrix}
\in \partial \left\| \begin{bmatrix}
1 & \tilde{X}_{12} \\
\tilde{X}_{21} & 0 
\end{bmatrix}\right\|_{*}.
$$%\end{equation}
This, together with Proposition~\ref{prop:x-in-sub}, shows that $\tilde{X}\in\Sy_+^2$. Since $\tilde{X}_{22}=0$, we have $\tilde{X}_{12} = \tilde{X}_{21} = 0$. It follows that $\tilde{X} = \bar{X}$, as claimed. In particular, Problem~\eqref{eq:counter-ex} satisfies Assumption~\ref{ass:level-bd} as well.

Now, let $\{\delta_k\}_{k\ge0}$ be a sequence such that $\delta_k\searrow0$ and define the sequence $\{X^k\}_{k\ge0}$ by
$$ X^k = \begin{bmatrix}
1 + 2\delta_k^2 & \delta_k \\
\delta_k & \delta_k^2
\end{bmatrix} \quad\mbox{for } k=0,1,\ldots.
$$
It is clear from the construction that $X^k\limto\bar{X}$ and
\begin{equation} \label{eq:dist-Xoptimal}
d(X^k,\mathcal{X}) = \|X^k-\bar{X}\|_F = \Theta(\delta_k).
\end{equation}
On the other hand, using~\eqref{eq:gradient-f}, we have
$$ \nabla f(X^k) = \begin{bmatrix}
-1 + \delta_k^2 & 0 \\
0 & -1-\delta_k^2
\end{bmatrix}.
$$
It follows that
\begin{equation} \label{eq:R(Xk)}
R(X^k) = S_1\left( X^k - \nabla f(X^k) \right) - X^k = S_1\left(\begin{bmatrix}
2 + \delta_k^2 & \delta_k \\ \delta_k & 1 + 2\delta_k^2
\end{bmatrix}\right) - \begin{bmatrix}
1 + 2\delta_k^2 & \delta_k \\
\delta_k & \delta_k^2
\end{bmatrix},
\end{equation}
where $S_1:\R^{m\times n}\limto\R^{m\times n}$ is the so-called {\it matrix shrinkage operator} defined as follows: Given a matrix $X\in\R^{m\times n}$, $S_1(X)\in\R^{m\times n}$ is the matrix obtained by taking any SVD of $X$ and replacing the singular value $\sigma_i(X)$ by $\max\{\sigma_i(X)-1,0\}$, for $i=1,\ldots,m$; see, \eg,~\cite[Theorem 3]{MGC11}.  Since
$$ \begin{bmatrix} 2 + \delta_k^2 & \delta_k \\ \delta_k & 1 + 2\delta_k^2
\end{bmatrix} \succeq I_2 \quad\mbox{for } k=0,1,\ldots, $$
the definition of $S_1$ implies that
$$ S_1\left(\begin{bmatrix}
2 + \delta_k^2 & \delta_k \\ \delta_k & 1 + 2\delta_k^2
\end{bmatrix}\right) = \begin{bmatrix}
2 + \delta_k^2 & \delta_k \\ \delta_k & 1 + 2\delta_k^2
\end{bmatrix} - I_2 = \begin{bmatrix}
1 + \delta_k^2 & \delta_k \\ \delta_k & 2\delta_k^2
\end{bmatrix}.
$$
Upon substituting the above equation into~\eqref{eq:R(Xk)}, we obtain
$$ R(X^k) = \begin{bmatrix}
-\delta_k^2 & 0 \\
0 & \delta_k^2 
\end{bmatrix} \quad\mbox{for } k=0,1,\ldots,
$$
which shows that $\|R(X^k)\|_F=\Theta(\delta_k^2)$. This, together with~\eqref{eq:dist-Xoptimal}, leads to $\|R(X^k)\|_F = o(d(X^k,\mathcal{X}))$. Consequently, Problem~\eqref{eq:counter-ex} does not possess the error bound property~\eqref{eq:err-bd}. It is worth noting that since $\bar{X}$ is the unique optimal solution to Problem~\eqref{eq:counter-ex} and
$$ -\nabla f(\bar{X}) = I_2 \not\in \mbox{ri}(\del\|\bar{X}\|_*) = \left\{ Z\in\R^{2\times 2} \mid Z_{11} = 1, \; Z_{12} = Z_{21} = 0, \; Z_{22}\in (-1,1) \right\}, $$
the regularity condition~\eqref{eq:nucnorm-cq} fails to hold in this example.

\section{Conclusion}
In this paper, we employed tools from set-valued analysis to develop a new framework for establishing error bounds for a class of structured convex optimization problems. We showed that such a framework can be used to recover a number of existing error bound results in a unified and transparent manner. To further demonstrate the power of our framework, we applied it to a class of nuclear-norm regularized loss minimization problems and showed, for the first time, that this class of problems possesses an error bound property under a strict complementarity-type regularity condition. We then complemented this result by constructing an example to show that the said error bound could fail to hold without the regularity condition.  Consequently, we obtained a rather complete answer to a question raised by Tseng~\cite{tseng2010approximation}. A natural and interesting future direction is to apply our framework to study the error bound property associated with other families of instances of Problem~\eqref{eq:str-cvx-prob} in which $P$ is non-polyhedral; see, \eg,~\cite{ZZS15}.

\section*{Acknowledgements}
We would like to express our gratitude to Professor Defeng Sun for his insightful comments on this work and for his constant encouragement.  We would also like to thank Professor Tom Luo for fruitful discussions and Professor Shaohua Pan for sending us the unpublished manuscript~\cite{PL15}.  This work is supported in part by the Hong Kong Research Grants Council (RGC) General Research Fund (GRF) Project CUHK 14206814 and in part by a gift grant from Microsoft Research Asia.

\bibliographystyle{abbrv}
\bibliography{eb-dft-v1}

\end{document}

%% file: defs.tex
\newcommand{\bz}{\mathbf 0}

\newcommand{\limto}{\rightarrow}
\newcommand{\del}{\partial}

\newcommand{\R}{\mathbb R}
\newcommand{\Or}{\mathbb O}
\newcommand{\Sy}{\mathbb S}

\newcommand{\argmin}{\mathop{\rm argmin}}

\newcommand{\cf}{{\it cf.}}
\newcommand{\eg}{{\it e.g.}}
\newcommand{\ie}{{\it i.e.}}

\newtheorem{lemma}{Lemma}
\newtheorem{thm}{Theorem}
\newtheorem{coro}{Corollary}
\newtheorem{defi}{Definition}
\newtheorem{prop}{Proposition}

\newtheorem{ass}{Assumption}
\newtheorem{fact}{Fact}

\newenvironment{proof}{{\bf Proof\,\,}}{\endproof\par}

\newcounter{spb}
\newcommand{\subpb}{(\alph{spb}) \addtocounter{spb}{1}}

\newcommand{\resetspb}{\setcounter{spb}{1}}

\def \openbox{$\sqcup\llap{$\sqcap$}$}
\def \endproof{\enskip \null \nobreak \hfill \openbox \par}